%% file: shear_arXiv2.tex
\documentclass[10pt,arxiv,scale=0.81]{agn_article}
\usepackage[tikzexternalize=false]{agn_all}

\newcommand{\TBiot}{T^{\mathrm{Biot}}}
\newcommand{\Wiso}{W_{\mathrm{iso}}}

\newcommand{\stress}{T}
\renewcommand{\stretch}{P}
\newcommand{\hTBiot}{\widehat T^{\mathrm{Biot}}}
\newcommand{\etss}{\mathrm{(E}\text{-}\mathrm{TSS)}}
\newcommand{\wetss}{\mathrm{(WE}\text{-}\mathrm{TSS)}}
\newcommand{\bep}{\mathrm{(BE^+)}}

\newcommand{\semi}{\mathrm{(semi)}}

\newcommand{\citet}[2][]{\citeauthor{#2} \cite[#1]{#2}}
\renewcommand{\Rp}{\R^+}
\newcommand{\sigmalin}{\sigma_{\rm lin}}
\newcommand{\TBiothat}{\widehat{T}^{\rm Biot}}
\newcommand{\Wsym}{W_{\mathrm{tc}}}
\newcommand{\changed}[1]
	{\textcolor{blue}{#1}}
\newcommand{\longchange}
	{\color{blue}}

\usepackage[font=small]{caption}

\begin{document}
\title{\vspace{-2cm}Shear, pure and simple}
\date{\today}
\author{%
	Christian Thiel\thanks{
		Christian Thiel,  \ \ Lehrstuhl f\"{u}r Nichtlineare Analysis und Modellierung, Fakult\"{a}t f\"{u}r Mathematik, Universit\"{a}t Duisburg-Essen,  Thea-Leymann Str. 9, 45127 Essen, Germany, email: christian.thiel@uni-due.de
		}\quad and\quad%
	Jendrik Voss\thanks{%
		Jendrik Voss,\quad Lehrstuhl f\"{u}r Nichtlineare Analysis und Modellierung, Fakult\"{a}t f\"{u}r Mathematik, Universit\"{a}t Duisburg-Essen, Thea-Leymann Str. 9, 45127 Essen, Germany; email: max.voss@uni-due.de%
	}\quad and\quad%
	Robert J.\ Martin\thanks{%
		Robert J.\ Martin,\quad Lehrstuhl f\"{u}r Nichtlineare Analysis und Modellierung, Fakult\"{a}t f\"{u}r Mathematik, Universit\"{a}t Duisburg-Essen, Thea-Leymann Str. 9, 45127 Essen, Germany; email: robert.martin@uni-due.de%
	}\quad and\quad%
	Patrizio Neff\thanks{%
		Corresponding author: Patrizio Neff,\quad Head of Lehrstuhl f\"{u}r Nichtlineare Analysis und Modellierung, Fakult\"{a}t f\"{u}r	Mathematik, Universit\"{a}t Duisburg-Essen, Thea-Leymann Str. 9, 45127 Essen, Germany, email: patrizio.neff@uni-due.de%
		}
}
\maketitle
\begin{abstract}
In a 2012 article in the International Journal of Non-Linear Mechanics, Destrade et al.\ showed that for nonlinear elastic materials satisfying Truesdell's so-called empirical inequalities, the deformation corresponding to a Cauchy pure shear stress is not a simple shear. Similar results can be found in a 2011 article of L.\ A.\ Mihai and A.\ Goriely. We confirm their results under weakened assumptions and consider the case of a shear load, i.e.\ a Biot pure shear stress. In addition, conditions under which Cauchy pure shear stresses correspond to (idealized) pure shear stretch tensors are stated and a new notion of idealized \emph{finite simple shear} is introduced, showing that for certain classes of nonlinear materials, the results by Destrade et al.\ can be simplified considerably.
\end{abstract}
\\[.65em]
\noindent {\bf AMS 2010 subject classification: 74B20, 74A20, 74A10}
{\parskip=-0.5mm \tableofcontents}
%
%
%
%
%%%%%%%%%%
\section{Introduction}
\label{section:introduction}
The term \enquote{shear} describes a number of closely related, but distinct concepts which play an important role in linear and nonlinear elasticity theory. While the notion of a \emph{(pure) shear stress} $T=s\.(e_1\otimes e_2+e_2\otimes e_1)$ with $s\in\R$ is rather straightforward \cite{bvelik1998state,boulanger2004pure} (once the stress tensor is specified), the concept of a \emph{shear deformation} is rather ambiguous: typically, the term \emph{(simple) shear (deformation)} is used to describe a specific homogeneous deformation gradient $\nabla\varphi=F\in\GLp(3)$ of the form $F=\id+\gamma\.e_1\otimes e_2$ with the amount of shear $\gamma\in\R$. On the other hand, a stretch tensor $P\in\PSym(3)$ is often called a \emph{(pure) shear (stretch)} if the principal stretches $\lambda_1,\lambda_2,\lambda_3$ are of the form $\lambda_1=\lambda$, $\lambda_2=\frac1\lambda$ and $\lambda_3=1$ with $\lambda>0$, see e.g.\ \cite{horgan2010simple,moreira2013comparison,nunes2013simple}. Hence, a pure shear stretch is a three-dimensional homogeneous deformation of a body which is elongated in one direction while being shortened in equal ratio perpendicularly.

The latter requirement is satisfied by the stretch corresponding to any simple shear deformation $F$ with $\lambda=\frac 12(\sqrt{\gamma^2 + 4} + \gamma)$. Since any pure shear stretch corresponds to a simple shear $F$ by appropriate left- and right-hand rotation, these two notions of shear often get confounded in the literature. In a 2012 article, Destrade et al.\ \cite{destrade2012} clarified the relation between Cauchy pure shear stresses and simple shear deformations, demonstrating their incompatibility in nonlinear elasticity. Similar observations can be found in an article by Moon and Truesdell \cite{moon1974interpretation} and a more current paper by Mihai and Goriely \cite{mihai2011positive}. In the following, we extend their discussion on the notion of shear in nonlinear elasticity. Undoubtedly, the primary concept in nonlinear elasticity is the force (stress) which causes the deformation. Therefore, our answer as to what a shear deformation is starts from considering Cauchy pure shear stress tensors.

\subsection{Overview}

After a short introduction and a brief discussion of the linear case in Section \ref{section:shearInLinearElasticity}, we demonstrate in Section \ref{section:shearInNonlinearElasticity} that non-trivial Cauchy pure shear stress tensors never correspond to simple shear deformations for arbitrary non-linear isotropic elasticity laws. This result was previously obtained by Destrade et al.\ \cite{destrade2012} in the special case of constitutive laws satisfying Truesdell's so-called \emph{empirical inequalities}. We then provide a number of representations for the class of deformations which do correspond to Cauchy pure shear stresses, including the composition of a triaxial stretch with a simple shear deformation as given in \cite{destrade2012,moon1974interpretation,mihai2011positive}. %}given by Destrade et al.\ \cite{destrade2012}, cf.\ \cite{moon1974interpretation,mihai2011positive}. 
In Appendix \ref{sec:Biot}, we extend these considerations to the case of Biot pure shear stress tensors and show that analogous results hold.

Based on this classification, we introduce constitutive requirements for hyperelastic laws which ensure that the principal stretches $\lambda_1,\lambda_2,\lambda_3$ corresponding to pure shear stresses are of the form $\lambda_1=\lambda$, $\lambda_2=\frac1\lambda$ and $\lambda_3=1$ and provide a representation of deformation gradients with this particular property in Sections \ref{section:finiteSimpleShear} and \ref{sec:Constitutive}.

While the results we obtain in Section \ref{section:shearInNonlinearElasticity} are not new \cite{moon1974interpretation,destrade2012,mihai2011positive}, we hope that by following a deliberately careful and detailed approach, we are able to clear up any confusion the reader might still face regarding the general relation between shear deformations and shear stresses in nonlinear elasticity. We also try to highlight the differences between the proofs given in \cite{moon1974interpretation}, \cite{destrade2012} and \cite{mihai2011positive}, especially with respect to their additional requirements on the constitutive law, cf.\ Remark \ref{remark:distinctEigenvaluesForPureShearStress} and Remark \ref{remark:differentProofs}.
%
%
%%%%%%%%%%
\subsection{Different notions of shear}

The classical (homogeneous) \emph{simple shear deformation} with the deformation gradient tensor $F=\id+\gamma\.e_2\otimes e_1$ of a unit cube with the amount of shear $\gamma\in\Rp=(0,\infty)$ is shown in Figure \ref{fig:simpleShearDeformation}.
It is well known that in isotropic linear elasticity, the Cauchy stress tensor corresponding to deformations of this type is necessarily of the form $\sigma=s\.(e_1\otimes e_2+e_2\otimes e_1)$ with $s\in\R$, known as \emph{pure shear stress}, as visualized in Figure \ref{fig_shear}.

\begin{figure}[h!]
    \centering
%   \tikzRemake
    \tikzsetnextfilename{simpleShearDeformation}
    \begin{minipage}{0.6\textwidth}
      \begin{tikzpicture}
        \input{tikz/tikz_simpleShearDeformation.tex}
      \end{tikzpicture}
    \end{minipage}
%    \begin{minipage}{0.1\textwidth}
%      \[
%        F=\nabla\varphi= \matr{1 & \gamma & 0 \\ 0 & 1 & 0 \\ 0 & 0 & 1}
%      \]
%    \end{minipage}
    \caption{Simple shear deformation with shear angle $\vartheta$ and amount of shear $\gamma=\tan(\vartheta)\,.$}		\label{fig:simpleShearDeformation}
\end{figure}
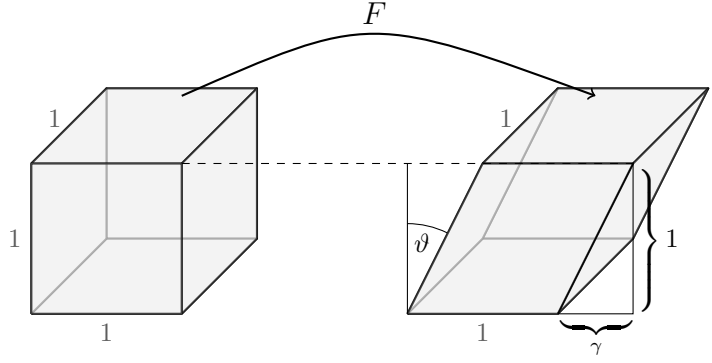

\begin{figure}[h]
    \centering
%   \tikzRemake
    \tikzsetnextfilename{shear}
    \begin{tikzpicture}
      \input{tikz/tikz_shear.tex}
    \end{tikzpicture}
    \caption{Visualization of pure shear stress. Right, Cauchy cube under pure shear stress in the current deformation.}\label{fig_shear}
\end{figure}
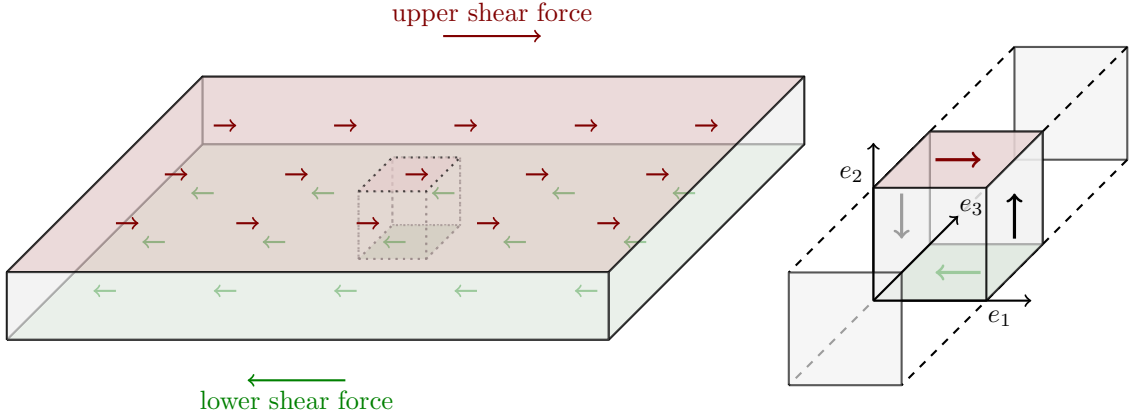
\fbox{\parbox{.98\textwidth}{
\begin{definitions}
\label{definition:linearSimpleShear}
	We call $F\in\GLp(3)$ a \emph{simple shear deformation} if $F$ has the form
	\begin{equation}\label{eq:definitionSimpleShear}
		F_\gamma = \matr{1&\gamma&0\\0&1&0\\0&0&1}=\id+\gamma\.e_2\otimes e_1\qquad\text{with $\gamma\in\R$ the amount of shear}
	\end{equation}
	and we call $T\in\Sym(3)$ \emph{pure shear stress}, if $T$ has the form
	\begin{equation}\label{eq:definitionPureShearStress}
		T^s = \matr{0&s&0\\s&0&0\\0&0&0}=s\.(e_1\otimes e_2+e_2\otimes e_1)\qquad\text{with $s \in\R$ the amount of shear stress.}
	\end{equation}
\end{definitions}}}
 
Here, $\GLp(3)=\{X\in\R^{3\times 3}\,|\,\det X>0\}$ is the group of invertible $3\times 3$ matrices with positive determinant and $\Sym(3)=\{X\in\R^{3\times 3}\,|\,X=X^T\}$ donates the group of symmetric matrices.
In the following, we will discuss the incompatibility of simple shear deformations and pure shear stress in nonlinear Cauchy elasticity \cite{moon1974interpretation,destrade2012,mihai2011positive}. As a substitute, we introduce the apparently new concept of (idealized) \emph{left} and \emph{right finite simple shear deformation} shown in Figure \ref{fig:finiteSimpleShearDeformation} as well as the (idealized) \emph{finite pure shear stretch}.\footnote{%
	The class of finite pure shear stretches is a multiplicative group \cite{beygelzimer2015equivalent}, also known as the \emph{group of hyperbolic rotations}. Note that $\exp\matr{0&\alpha\\\alpha&0}* \exp\matr{0&\beta\\\beta&0}=\exp\matr{0&\alpha+\beta\\\alpha+\beta&0}$ since $\matr{0&\alpha\\\alpha&0}$, $\matr{0&\beta\\\beta&0}$ commute, cf.~\cite[p.~39]{olver1995equivalence} and Figure \ref{fig:pureShearStretch}.
	To our knowledge, the finite pure shear stretch \eqref{eq:definitionPureShearStretch} was first mentioned by Claude Vall{\'e}e \cite{vallee1978} as the stretch induced by a Hencky-type logarithmic stress-strain relation under Cauchy pure shear stress.
}

\fbox{\parbox{.98\textwidth}{
\begin{definitions}\label{def:finiteSimpleShear}
	We call $F\in\GLp(3)$ an (idealized) \emph{left finite simple shear deformation gradient} if $F$ has the form
	\begin{equation}\label{eq:definitionLeftSimpleFiniteShear}
		F_\alpha = \frac{1}{\sqrt{\cosh(2\.\alpha)}}\, \matr{1&\sinh(2\.\alpha)&0\\ 0&\cosh(2\.\alpha)&0\\ 0&0&\sqrt{\cosh(2\.\alpha)}}\qquad\text{with $\alpha\in\R$}
	\end{equation}
	and an (idealized) \emph{right finite simple shear deformation gradient} if $F$ has the form
	\begin{equation}\label{eq:definitionRightSimpleFiniteShear}
		F_\alpha =  \frac{1}{\sqrt{\cosh(2\.\alpha)}}\, \matr{\cosh(2\.\alpha)&\sinh(2\.\alpha)&0\\ 0&1&0\\ 0&0&\sqrt{\cosh(2\.\alpha)}}\qquad\text{with $\alpha\in\R$.}
	\end{equation}
	Moreover, we call $V\in\PSym(3)$ a \emph{finite pure shear stretch} if $V$ has the form
	\begin{equation}\label{eq:definitionPureShearStretch}
		V_\alpha =  \matr{\cosh(\alpha)&\sinh(\alpha)&0\\\sinh(\alpha)&\cosh(\alpha)&0\\0&0&1}=\exp\matr{0&\alpha&0\\\alpha&0&0\\0&0&0}\qquad\text{with $\alpha\in\R$}\
	\end{equation}
	where $\exp$ donates the matrix exponential.
\end{definitions}}}

Lemma \ref{lemma:polarFiniteShear} states that for given $ \alpha \in \R $, the left and the right finite simple shear deformation are composed of exactly the same stretch and rotation, but applied in different order. Both the left and right finite simple shear deformations can be considered as finite extensions of the simple shear deformation due to their linearization (see Appendix \ref{appendix:LinearShear}); the canonical identification of the amount of shear in the infinitesimal case is $\gamma=2\.\alpha$.

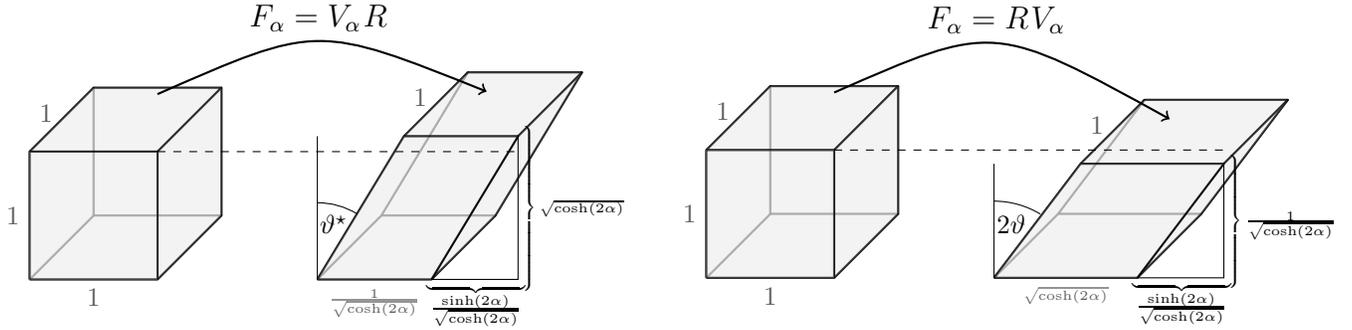
\begin{figure}[h!]
    \centering
%    \tikzRemake
    \tikzsetnextfilename{FiniteSimpleShearDeformation}  
    \begin{minipage}{0.4\textwidth}
       	\hspace{-2cm} \begin{tikzpicture}[scale=0.85]
        \input{tikz/tikz_leftFiniteSimpleShearDeformation.tex}
      \end{tikzpicture}
    \end{minipage}
    \begin{minipage}{0.4\textwidth}
      \begin{tikzpicture}[scale=0.85]
        \input{tikz/tikz_rightFiniteSimpleShearDeformation.tex}
      \end{tikzpicture}
    \end{minipage}
    \caption{(left figure:) Left finite simple shear deformation and (right figure:) right finite simple shear deformation with $\tan(\vartheta^\star)=\tanh(2\alpha)$ and $\tan(\vartheta) = \tanh(\alpha)$.}\label{fig:finiteSimpleShearDeformation}
\end{figure} 
\begin{figure}[h!]
    \centering
%    \tikzRemake
    \tikzsetnextfilename{pureShearStretch}
    \begin{minipage}{0.6\textwidth}
      \begin{tikzpicture}[scale=0.85]
        \input{tikz/tikz_pureShearStretch.tex}
      \end{tikzpicture}
    \end{minipage}
    \caption{Finite pure shear stretch with $\tan(\vartheta) = \tanh(\alpha)$.}\label{fig:pureShearStretch}
\end{figure}
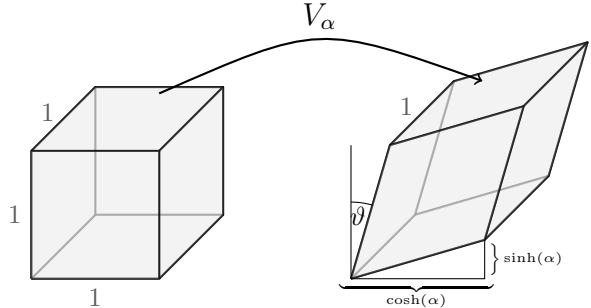
%
%
%
%
%%%%%%%%%%
\subsection{Shear in linear elasticity}\label{section:shearInLinearElasticity}
In isotropic linear elasticity, the stress response is induced by the quadratic energy density function $\Wlin(\eps)=\mu\.\norm{\eps}^2+\frac{\lambda}{2}\tr(\eps)^2=\mu\.\norm{\dev\eps}^2+\frac{\kappa}{2}\tr(\eps)^2$, where $\eps = \sym(F-\id)$, $\dev(X)=X-\frac{1}{3}\tr(X)\.\id$; here, $\mu$ denotes the infinitesimal shear modulus, $\kappa>0$ is the infinitesimal bulk modulus, and $\lambda$ is the first Lam\'e parameter. The corresponding linear Cauchy stress tensor is given by $\sigmalin = 2\.\mu\.\dev\eps + \kappa\.\tr(\eps)\.\id=2\.\mu\.\eps+\lambda\tr(\eps)\.\id\,.$ The following proposition describes the well-known relation between simple shear deformations and pure shear stress in the isotropic linear elasticity model.
\begin{proposition}
	The linear Cauchy shear stress $\sigmalin$ induced by a deformation $F\in\GLp(3)$ is a pure shear stress if and only if the deformation has the form
	\setlength{\belowdisplayskip}{0pt}
	\begin{equation}\label{eq_form_von_grad_u}
		F = \id + \underbrace{\begin{pmatrix} 0 & \frac \gamma 2 & 0 \\ \frac \gamma 2 & 0 & 0 \\ 0 & 0 & 0\end{pmatrix}}_{\eps} + A\qquad\text{with $\gamma\in\R$ and arbitrary $A\in\so(3)$.}
	\end{equation}
\end{proposition}
\begin{proof}
	The relation $\sigmalin = 2\.\mu\.\dev\eps + \kappa\.\tr(\eps)\.\id$ immediately implies that a deformation of the form \eqref{eq_form_von_grad_u} induces an infinitesimal Cauchy pure shear stress tensor with the amount of shear stress $s=\mu\gamma$.
	
	On the other hand, if $\sigmalin$ is a pure shear stress, then $0=\tr(\sigmalin)=\tr\left(2\mu\.\dev\eps+\kappa\.\tr(\eps)\.\id\right)=3\kappa\.\tr(\eps)$. Thus $\sigmalin=2\.\mu\.\eps$, i.e.
	\[
		2\.\mu\.\eps=\matr{0 & s & 0 \\ s & 0 & 0 \\ 0 & 0 & 0}
		\qquad\implies\qquad
		\eps = \matr{ 0 & \frac s{2\mu} & 0 \\ \frac s{2\mu} & 0 & 0 \\ 0 & 0 & 0}\,,
	\]
	which leads to a deformation of the form \eqref{eq_form_von_grad_u} with $\gamma = \frac s\mu$ and arbitrary $A\in\so(3)$.
\end{proof}
\begin{remark}
	Every simple shear deformation \eqref{eq:definitionSimpleShear} has the form \eqref{eq_form_von_grad_u} and therefore leads to an infinitesimal pure shear stress tensor $\sigmalin$:
	\setlength{\belowdisplayskip}{0pt}
	\[
		F_\gamma= \matr{1&\gamma&0\\0&1&0\\0&0&1}=\id+\matr{0&\gamma&0\\0&0&0\\0&0&0} = \id + \underbrace{\begin{pmatrix} 0 & \frac \gamma 2 & 0 \\ \frac \gamma 2 & 0 & 0 \\ 0 & 0 & 0\end{pmatrix}}_{\eps_\gamma\.\in\Sym(3)} + \underbrace{\matr{ 0 & \frac \gamma 2 & 0 \\ -\frac \gamma 2 & 0 & 0 \\ 0 & 0 & 0 }}_{\omega_\gamma\.\in\so(3)}=\id+\eps_\gamma+\omega_\gamma\qquad\text{with $\gamma\in\R$.} \qedhere
	\]
\end{remark}
Thus it is possible to \textbf{additively split} every deformation of the form \eqref{eq_form_von_grad_u}, which leads to a pure shear stress in the isotropic linear model, as follows (cf.\ Figure \ref{fig_split_simpleShearDeformation}):\par
\fbox{\parbox{.98\textwidth}{
\begin{definitions}\label{def:linearShear}
	We call 
	$\eps\in\Sym(3)$ an \emph{infinitesimal pure shear strain} if $\eps$ has the form
	\begin{equation}\label{eq:definitionLinearShear}
		\eps_\gamma = \matr{0&\frac{\gamma}{2}&0\\\frac{\gamma}{2}&0&0\\0&0&0}\qquad\text{with $\gamma\in\R$ the amount of shear,}
	\end{equation}
	and $\omega\in\so(3)$ an \emph{infinitesimal rotation of simple shear} if $\omega$ has the form
	\begin{equation}
		\omega_\gamma = \matr{0&\frac{\gamma}{2}&0\\-\frac{\gamma}{2}&0&0\\0&0&0}\qquad\text{with $\gamma\in\R$ the amount of shear.}
	\end{equation}
	In particular, the finite pure shear stretch $V_\alpha$ \eqref{eq:definitionPureShearStretch} is the matrix exponential extension of the infinitesimal pure shear strain $\eps_\gamma$, i.e.\ $V_\alpha=\exp(\eps_\gamma)$ with $\gamma=2\.\alpha$.
\end{definitions}}}\par
\begin{figure}[h!]
    \centering
%   \tikzRemake
    \tikzsetnextfilename{DecompositionSimpleShear}
    \begin{tikzpicture}
        \input{tikz/tikz_DecompositionSimpleShear.tex}
    \end{tikzpicture}
    \caption{Additive split of the simple shear deformation $F_\gamma$ into infinitesimal pure shear strain $\eps_\gamma$ and infinitesimal rotation $\omega_\gamma$.}\label{fig_split_simpleShearDeformation}
\end{figure}
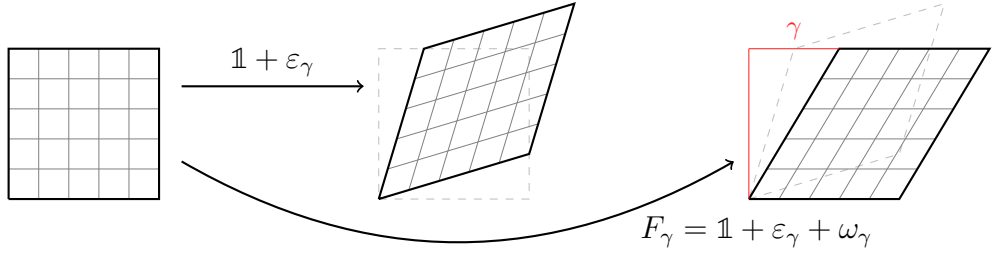

We observe that $F_\gamma=\id+\eps_\gamma+\omega_\gamma$ is of the form \eqref{eq:definitionSimpleShear} and that
\begin{itemize}
	\item[i)]  the deformation $F_\gamma$ is \emph{infinitesimally volume preserving}, i.e\ $\tr(F_\gamma-\id)=0$,
	\item[ii)] the deformation $F_\gamma$ is \emph{planar}, i.e.\ $F_\gamma$ has the eigenvalue $1$ to the eigenvector $e_3$,
	\item[iii)] the deformation $F_\gamma$ is \emph{ground parallel}, i.e.\ $e_1$ and $e_3$ are eigenvectors of $F_\gamma$.
\end{itemize}
\begin{remark}
\label{remark:linearStretchNotFinitelyVolumePreserving}
	Note carefully that the property $\det F_\gamma=1$ is of no consequence in linear elasticity since, in contrast to the finite case, the infinitesimal change of volume induced by the deformation gradient $F$ is described by $\tr(\eps)=\tr(F)-3$ and not the determinant of $F$. In particular, for $\gamma\neq0$, the corresponding infinitesimal pure shear strain $\eps_\gamma = \sym(F_\gamma-\id)$ is \emph{not} finitely volume preserving, i.e.\ $\det(\id+\eps_\gamma)\neq1$.
	
	The \enquote{accidental} fact that the simple shear deformation \eqref{eq:definitionSimpleShear} is volume preserving in the finite sense as well appears to be a major source of confusion, since it suggests that in nonlinear elasticity, a \enquote{shear} deformation should have the exact form $F_\gamma$. However, as the result by Destrade et al.\ \cite{destrade2012} clearly demonstrates, the finite generalization of simple shear deformations must take another form. In Section \ref{section:finiteSimpleShear}, we will show that deformations of the idealized form \eqref{eq:definitionLeftSimpleFiniteShear} are not only compatible with the linearized shear $F_\gamma$ and generalize the properties i)--iii), but can be compatible with Cauchy pure shear stresses as well.
\end{remark}
%
%
%
%%%%%%%%%%
\section{Shear in nonlinear elasticity}
\label{section:shearInNonlinearElasticity}
The following theorem summarizes the aforementioned result by Destrade et al.\ \cite{destrade2012}.
\begin{theorem}[Destrade, Murphy and Saccomandi \cite{destrade2012}]\label{theorem:destrade}
	Consider an isotropic elasticity law that satisfies the empirical inequalities \cite{truesdell1952,truesdell65,truesdell1975inequalities,moon1974interpretation,truesdell1963static}
	\begin{align}
		\beta_0\leq 0\,,\qquad\beta_1>0\,,\qquad\beta_{-1}\leq 0\qquad\text{with}\qquad\sigma=\beta_0\.\id+\beta_1\.B+\beta_{-1}\.B^{-1}\,,\label{eq:empiricalInequalities}
	\end{align}
	also called the \textbf{e}mpirical \textbf{t}rue \textbf{s}tress \textbf{s}tretch $\etss$-relation \cite{agn_neff2015exponentiatedI}. If, for $F\in\GL(3)$, $\widehat\sigma(FF^T)=\sigma^s$ is a Cauchy pure shear stress, then there exists a triaxial stretch $\diag(a,b,c)$ and a simple shear deformation $F_\gamma$ such that $F = F_\gamma\cdot\diag(a,b,c)\cdot Q$ for some $Q\in\SO(3)$.
	Furthermore, if $s\neq0$ for the amount of stress $s$, then $\diag(a,b,c)\neq\id$, i.e.\ a simple shear deformation $F_\gamma\neq\id$ alone never corresponds to a nontrivial Cauchy pure shear stress of the form \eqref{eq:definitionPureShearStress}.
\end{theorem}
Destrade et al.\ implicitly utilize the so-called \emph{semi-invertibility} \cite{truesdell1975inequalities} of the Cauchy stress response function $\sigmahat\col FF^T=B\mapsto\sigmahat(B)=\sigma(\sqrt{FF^T})$, i.e.\ the representability
\begin{equation}\label{eq_repraesentationsformel}
	B = \psi_0(B)\.\id + \psi_1(B)\.\sigma + \psi_2(B)\.\sigma^2\qquad\text{for all $B\in\PSym(3)$}
\end{equation}
which is ensured\footnote{It should be noted that the much weaker condition of the strict Baker-Ericksen inequalities $\bep$ \cite{baker1954inequalities} would have been sufficient as well \cite{dunn1984elastic}, cf.\ the implication chain \enquote{$\etss \implies \wetss \implies \bep \implies \semi$} from \cite[p. 135]{thiel2017neue}.} by the empirical inequalities \eqref{eq:empiricalInequalities}. Note that \eqref{eq_repraesentationsformel} does not imply that the Cauchy stress-stretch law is invertible since the functions $\psi_i$ still depend on $B$.

\begin{remark}
	It is important to note that the condition of semi-invertibility, whether it is assumed directly or indirectly, always restricts the considered class of elasticity laws. In the following, we waive this restriction do not require any constitutive conditions to hold except the isotropy of the elasticity law and the uniqueness of the stress-free reference configuration, i.e.\ $\widehat\sigma(B) = 0$ if and only if $B = \id$.
\end{remark}
\begin{remark}
	A more general version of this result stated in Corollary \ref{corollary:destradeWir} was given by Moon and Truesdell \cite[(2.8)]{moon1974interpretation} and, implicitly, by Mihai and Goriely, cf.\ Remark \ref{remark:differentProofs}.
\end{remark}
Our computations are similar in part to those by Destrade et al.; however, instead of requiring the semi-invertibility \eqref{eq_repraesentationsformel}, we utilize the commutativity of the left Cauchy-Green deformation tensor $B=FF^T$ with the Cauchy stress tensor $\widehat\sigma(B)$ which holds for \textbf{any} isotropic stress response. Thereby, in \reflemma{lem_P_und_T_kommutieren}, we are able to determine the general form of all $B$ which commute
with a Cauchy pure shear stress which, in turn, allows us to compute the general form of the deformation gradient $F = F_\gamma\cdot\diag(a,b,c)\cdot Q$ in \refproposition{prop_F_eindeutig_bestimmt_B}. For a more detailed description regarding semi-invertibility and related matrix properties, see \cite{agn2018empirical}.

\begin{proposition}\label{prop_t_dianoalgestalt}
	Let $T\in\PSym(3)$ be a pure shear stress of the form \eqref{eq:definitionPureShearStress}. Then $T$ has the eigenvalues $s,-s,0$ with the corresponding eigenvectors $(\frac{\sqrt 2}2, \frac{\sqrt 2}2, 0)^T$, $(-\frac{\sqrt 2}2, \frac{\sqrt 2}2, 0)^T$ and $(0,0,1)^T$. In particular, $T$ can be diagonalized to $Q\.\diag(s,-s,0)\.Q^T=T$ with
	\begin{equation}\label{eq_Q}
		Q \colonequals \frac {\sqrt 2}2\begin{pmatrix} 1 & -1 & \;\;\,0 \\ 1 & \phantom{-}1 & \;\;\,0 \\ 0 & \phantom{-}0 & \sqrt 2 \end{pmatrix}\in \SO(3)\,.
	\end{equation}
\end{proposition}
\begin{proof}
	We simply check
	\begin{align*}
		Q\begin{pmatrix} s & \phantom{-}0 & 0 \\ 0 & -s & 0 \\ 0 & \phantom{-}0 & 0 \end{pmatrix}Q^T &= \frac 12\begin{pmatrix} 1 & -1 & \;\;\,0 \\ 1 & \phantom{-}1 & \;\;\,0 \\ 0 & \phantom{-}0 & \sqrt 2 \end{pmatrix} \begin{pmatrix} s & \phantom{-}0 & 0 \\ 0 & -s & 0 \\ 0 & \phantom{-}0 & 0 \end{pmatrix} \begin{pmatrix} \phantom{-}1 & 1 & \;\;\,0 \\ -1 & 1 & \;\;\,0 \\ \phantom{-}0 & 0 & \sqrt 2 \end{pmatrix}=\frac 12\begin{pmatrix} 0 & 2s & 0 \\ 2s & 0 & 0 \\ 0 & 0 & 0 \end{pmatrix} = T\,.\qedhere
	\end{align*}
\end{proof}
In particular, a pure shear stress tensor is either trivial (the case of the stress-free zero tensor) or has \textbf{three distinct eigenvalues}.

Let $\widehat\sigma(B)$ be a pure shear stress for arbitrary $B=FF^T$ with $F\in\GLp(3)$. Since the assumed isotropy ensures that $B$ must commute with $\widehat\sigma(B)$, we consider the following lemma in order to deduce the specific form of $B$.

\begin{lemma}\label{lem_P_und_T_kommutieren}
	Let $T$ be a pure shear stress tensor with $s\neq 0$ and $P\in\PSym(3)$. Then $P$ commutes with $T$ if and only if $P$ has the form
	\begin{equation}
		P = \begin{pmatrix} p & q & 0 \\ q & p & 0 \\ 0 & 0 & r \end{pmatrix}\,.\label{form_von_P}
	\end{equation}
	Furthermore, $p=\frac 12(\mu_1+\mu_2)$, $q=\frac 12(\mu_1-\mu_2)$ and $r=\mu_3$, where $\mu_1,\mu_2,\mu_3 \in \Rp$ are the eigenvalues of $P$.
\end{lemma}
\begin{proof}
	Suppose $P$ and $T$ commute. Then $P$ and $T$ are simultaneously diagonalizable \cite[prop. 2.1]{agn2018empirical}. Furthermore, if $T$ has distinct eigenvalues, then $T$ is coaxial with $P$, i.e.\ each eigenvector of $T$ is an eigenvector of $P$. According to \refproposition{prop_t_dianoalgestalt}, $T$ can be diagonalized using $Q$ as in \eqref{eq_Q} and that the eigenvalues of $T$ are distinct if and only if $s\neq 0$, which holds by assumption. Therefore, $Q^TPQ$ must be in diagonal form as well. With $\mu_1,\mu_2,\mu_3\in\R$ denoting the eigenvalues of $P$ we find
	\begin{align*}
		P &=  Q\begin{pmatrix} \mu_1 & 0 & 0 \\ 0 & \mu_2 & 0 \\ 0 & 0 & \mu_3 \end{pmatrix}Q^T = \frac 12\begin{pmatrix} 1 & -1 & \;\;\,0 \\ 1 & \phantom{-}1 & \;\;\,0 \\ 0 & \phantom{-}0 & \sqrt 2 \end{pmatrix} \begin{pmatrix}\mu_1 & 0 & 0 \\ 0 & \mu_2 & 0 \\ 0 & 0 & \mu_3 \end{pmatrix} \begin{pmatrix} \phantom{-}1 & 1 & \;\;\,0 \\ -1 & 1 & \;\;\,0 \\ \phantom{-}0 & 0 & \sqrt 2 \end{pmatrix}\\
		&= \frac 12\begin{pmatrix} 1 & -1 & \;\;\,0 \\ 1 & \phantom{-}1 & \;\;\,0 \\ 0 & \phantom{-}0 & \sqrt 2 \end{pmatrix}\begin{pmatrix} \phantom{-}\mu_1 & \mu_1 & 0 \\ -\mu_2 & \mu_2 & 0 \\ 0 & 0 & \sqrt 2\mu_3\end{pmatrix}= \frac 12\begin{pmatrix} \mu_1 + \mu_2 & \mu_1 - \mu_2 & 0 \\ \mu_1 - \mu_2 & \mu_1 + \mu_2 & 0 \\ 0 & 0 & 2\mu_3\end{pmatrix}\,,
	\end{align*}
	thus $P$ is of the form \eqref{form_von_P}.

	Conversely, if $P$ is of the form \eqref{form_von_P}, then
	\[
		P\.T = \begin{pmatrix} p & q & 0 \\ q & p & 0 \\ 0 & 0 & r \end{pmatrix}\begin{pmatrix} 0 & s & 0 \\ s & 0 & 0 \\ 0 & 0 & 0 \end{pmatrix} = \begin{pmatrix} qs & ps & 0 \\ ps & qs & 0 \\ 0 & 0 & 0 \end{pmatrix} = \begin{pmatrix} 0 & s & 0 \\ s & 0 & 0 \\ 0 & 0 & 0 \end{pmatrix}\begin{pmatrix} p & q & 0 \\ q & p & 0 \\ 0 & 0 & r \end{pmatrix} = T\.P\,.\qedhere
	\]
\end{proof}
\begin{remark}\label{remark:distinctEigenvaluesForPureShearStress}
	Note that the proof of Lemma \ref{lem_P_und_T_kommutieren} makes use of the fact that a (non-trivial) pure shear stress has only simple eigenvalues; otherwise, it would not be possible to conclude that each eigenvector of $\stress$ is an eigenvector of $\stretch$ without additional assumptions. 
	
	In terms of elasticity tensors, this observation can be stated as follows: although for an isotropic material, each principal axis of (Eulerian) strain must be a principal axis of (Cauchy) stress, the reverse must not hold in general unless the principal stresses are pairwise distinct or additional constraints on the constitutive law are assumed to hold. Among the constitutive requirements which guarantee this \enquote{bi-coaxiality} of stress and strain are the (weak) empirical inequalities \cite{truesdell1956ungeloste,truesdell1975inequalities,truesdell1963static}, although the (weaker) strict Baker-Ericksen inequalities $\bep$ are sufficient as well \cite{dunn1984elastic}.  
	Moreover, \enquote{bi-coaxiality} is equivalent to semi-invertibility \eqref{eq_repraesentationsformel}. Again, recall the commutativity of the left Cauchy-Green deformation tensor $B$ with the Cauchy stress tensor $\sigmahat(B)$ and of the right Cauchy-Green deformation tensor $C$ with the Biot stress tensor $\TBiot(C)$ for any isotropic stress response, cf.\ \cite[p.193, Theorem 4.2.4]{Ogden83}.
\end{remark}
According to Lemma \ref{lem_P_und_T_kommutieren}, in order for $B = FF^T$ or $C = F^TF$ to commute with a pure shear stress tensor $T$, the Cauchy-Green deformation tensors $B,C$ must be of the form
\begin{equation}
	P=\begin{pmatrix} p & q & 0 \\ q & p & 0 \\ 0 & 0 & r \end{pmatrix}\, = \frac 12\begin{pmatrix} \lambda_1^2 + \lambda_2^2 & \lambda_1^2 - \lambda_2^2 & 0 \\ \lambda_1^2 - \lambda_2^2 & \lambda_1^1 + \lambda_2^2 & 0 \\ 0 & 0 & 2\lambda_3^2 \end{pmatrix}\label{eq:PinEigenvalues}
\end{equation}
with arbitrary singular values $\lambda_1,\lambda_2,\lambda_3$ of $F$, i.e.\ eigenvalues $\lambda_1^2,\lambda_2^2,\lambda_3^2$ of $B$ and $C$. Moreover,
\begin{equation}
	\sqrt{P}=\begin{pmatrix} p & q & 0 \\ q & p & 0 \\ 0 & 0 & r \end{pmatrix}^{\mathrlap{\!\nicefrac 12}}\, = \frac 12\begin{pmatrix} \lambda_1 + \lambda_2 & \lambda_1 - \lambda_2 & 0 \\ \lambda_1 - \lambda_2 & \lambda_1 + \lambda_2 & 0 \\ 0 & 0 & 2\lambda_3 \end{pmatrix}\,.\label{eq:squarerootGeneralForm}
\end{equation}

Using this representation, we can now determine the general form of all $F\in\GLp(3)$ (independent of the particular elasticity law) which are able to correspond to a Cauchy-Green tensor of the form \eqref{eq:PinEigenvalues} and thus to a Cauchy pure shear stress. In particular, we can confirm that $F$ cannot have the form of a simple shear deformation:

\begin{remark}
\label{remark:simpleShearBC}
	Let $F$ be a simple shear deformation gradient with $\gamma\neq 0$ and $T$ a pure shear stress tensor with $s\neq 0$. Because
	\begin{equation}
		F F^T = \begin{pmatrix} 1 & \gamma & 0 \\ 0 & 1 & 0 \\ 0 & 0 & 1 \end{pmatrix} \begin{pmatrix} 1 & 0 & 0 \\ \gamma & 1 & 0 \\ 0 & 0 & 1 \end{pmatrix} = \begin{pmatrix} 1 + \gamma^2 & \gamma & 0 \\ \gamma & 1 & 0 \\ 0 & 0 & 1 \end{pmatrix}
	\end{equation}
	and
	\begin{equation}
		F^T F = \begin{pmatrix} 1 & 0 & 0 \\ \gamma & 1 & 0 \\ 0 & 0 & 1 \end{pmatrix} \begin{pmatrix} 1 & \gamma & 0 \\ 0 & 1 & 0 \\ 0 & 0 & 1 \end{pmatrix} = \begin{pmatrix} 1 & \gamma & 0 \\ \gamma & 1 + \gamma^2 & 0 \\ 0 & 0 & 1 \end{pmatrix}
	\end{equation}
	are both not of the form \eqref{form_von_P}, neither $B=FF^T$ or $C=F^TF$ commutes with $T$.
\end{remark}
Instead, up to an arbitrary pure rotation, $F$ must take the form of a simple shear composed with a \emph{triaxial stretch}. Now let $\sigmahat(B)$ be a Cauchy pure shear stress. Then due to Lemma \ref{lem_P_und_T_kommutieren}, the left Cauchy-Green deformation tensor $B$ must be of the form \eqref{form_von_P}, which allows us to compute the general form of $F$ itself:

\begin{proposition}\label{prop_F_eindeutig_bestimmt_B}
	Let $B\in\PSym(3)$ be given by \eqref{form_von_P} with $p>\abs{q}$ and $r>0$. Then $F \in \GLp(3)$ with $FF^T=B$ is uniquely determined by
	\begin{equation}\label{eq:general_form}
		F =F_\gamma\.\diag(a,b,c)\,Q= \begin{pmatrix} 1 & \gamma & 0 \\ 0 & 1 & 0 \\ 0 & 0 & 1 \end{pmatrix} \begin{pmatrix} a & 0 & 0 \\ 0 & b & 0 \\ 0 & 0 & c \end{pmatrix}\,Q\,,
	\end{equation}
	up to an arbitrary $Q\in\SO(3)$, where
	\begin{align}
		a = \sqrt{\frac{p^2-q^2}p}\,,\qquad b = \sqrt{p} = \sqrt{B_{11}}\,,\qquad c = \sqrt{r}\qquad\text{and}\quad \gamma = \frac qp\label{eq:abcInpq}
	\end{align}
	or, in terms of the singular values $\lambda_1,\lambda_2,\lambda_3\in\R_+$ of $F$,
	\begin{align}
		a=\lambda_1\lambda_2\sqrt{\frac 2{\lambda_1^2+\lambda_2^2}} \,,\qquad b=\sqrt{\frac{\lambda_1^2 + \lambda_2^2}2}\,,\qquad c=\lambda_3\,,\qquad \gamma=\frac{\lambda_1^2-\lambda_2^2}{\lambda_1^2+\lambda_2^2}\,.\label{eq:abcInlambda}
	\end{align}
	In particular, $F$ is necessarily of the form \eqref{eq:general_form} if the Cauchy stress tensor corresponding to the deformation gradient $F$ induced by an isotropic law of Cauchy elasticity is a pure shear stress tensor of the form $\sigma=s\.(e_1\otimes e_2+e_2\otimes e_1)$ with $s\in\R$.
\end{proposition}
\begin{proof}
	Let $\widetilde{F}=F_\gamma\.\diag(a,b,c)$ with $a,b,c$ given by \eqref{eq:abcInpq}. Then $a^2 + b^2\gamma^2 = b^2 = p$, $b^2\gamma = q$ and $c^2=r$, thus
	\begin{align*}
		\widetilde F\widetilde F^T &=  \begin{pmatrix} 1 & \gamma & 0 \\ 0 & 1 & 0 \\ 0 & 0 & 1 \end{pmatrix} \begin{pmatrix} a & 0 & 0 \\ 0 & b & 0 \\ 0 & 0 & c \end{pmatrix}\begin{pmatrix} a & 0 & 0 \\ 0 & b & 0 \\ 0 & 0 & c \end{pmatrix} \begin{pmatrix} 1 & 0 & 0 \\ \gamma & 1 & 0 \\ 0 & 0 & 1 \end{pmatrix}= \begin{pmatrix} a^2+b^2\gamma^2 & b^2\gamma & 0 \\ b^2\gamma & b^2 & 0 \\ 0 & 0 & c^2 \end{pmatrix} = \begin{pmatrix} p & q & 0 \\ q & p & 0 \\ 0 & 0 & r \end{pmatrix}\,.
	\end{align*}

	Due to \reflemma{lem_P_und_T_kommutieren}, $p=\frac 12(\lambda_1^2+\lambda_2^2)$, $q=\frac 12(\lambda_1^2-\lambda_2^2)$ and $r=\lambda_3^2$, which immediately implies $b=\sqrt{\frac{\lambda_1^2 + \lambda_2^2}2}$, $c=\lambda_3$ and $\gamma=\frac{\lambda_1^2-\lambda_2^2}{\lambda_1^2+\lambda_2^2}$. Moreover, $p+q=\lambda_1^2$ and $p-q=\lambda_2^2$, thus $p^2-q^2=\lambda_1^2\,\lambda_2^2$ and therefore $a=\lambda_1\lambda_2\sqrt{\frac 2{\lambda_1^2+\lambda_2^2}}$.
	
	Now let $B$ be given by \eqref{form_von_P} and consider an arbitrary $F\in\GL^+(3)$ with $FF^T=B=\widetilde F\widetilde F^T$. Since $F$ is uniquely determined by $FF^T$ up to a right-hand rotation, there exists $Q\in\SO(3)$ with $F=\widetilde F\.Q$, thus $F$ is of the form \eqref{eq:general_form}.
\end{proof}
The following generalization of Theorem \ref{theorem:destrade} summarizes the above considerations.
\begin{corollary}\label{corollary:destradeWir}
	Consider an isotropic elasticity law with a unique stress-free reference configuration, i.e.\ $\widehat\sigma(B) = 0$ if and only if $B = \id$. Then any deformation gradient $F$ which induces a Cauchy pure shear stress $\sigma=s\.(e_1\otimes e_2+e_2\otimes e_1)$ with $s\in\R$ is necessarily of the form \eqref{eq:general_form}.
\end{corollary}
\begin{proof}
	Due to the assumed isotropy, the stretch $B=FF^T$ and the stress $\sigmahat(B)$ commute. According to Lemma \ref{lem_P_und_T_kommutieren}, $B$ must therefore be of the form \eqref{form_von_P}, thus $F$ must have the form \eqref{eq:general_form} due to Proposition \ref{prop_F_eindeutig_bestimmt_B}.
\end{proof}
\begin{remark}\label{rem_gamma}
In order for the term $\sqrt{\frac{p^2-q^2}p}$ to be well defined, the condition $p>|q|$ must hold. This implies the upper bound $|\gamma|=\frac{|q|}p < 1$, i.e.\ the shear angle is always limited by $45^\circ$. Note carefully that this limitation $p=\frac{1}{2}\.(\mu_1+\mu_2)>\frac{1}{2}\abs{\mu_1-\mu_2}=q$ is due to the positive definiteness of $B$ and not dependent on any constitutive requirements.
\end{remark}
\begin{remark}
\label{remark:differentProofs}
	While (unlike Moon and Truesdell \cite{moon1974interpretation}) Mihai and Goriely \cite{mihai2011positive} do not state the result of Corollary \ref{corollary:destradeWir} in full generality, their proof does not require the empirical inequalities, the Baker-Ericksen inequalities or the semi-invertibility of the Cauchy stress response. Both the proofs given in \cite{moon1974interpretation} and \cite{mihai2011positive} implicitly utilize the bi-coaxiality of the Cauchy stress tensor $\sigma$ and the stretch tensor $B$ and thus the distinctness of the eigenvalues, although they involve a more direct computation of the matrix entries of $B$ instead of the diagonalization approach presented here.
\end{remark}
The factor $b= \sqrt{B_{11}}$ in Proposition \ref{prop_F_eindeutig_bestimmt_B} plays an important role for the so-called \emph{Poynting effect} \cite{mihai2011positive,mihai2017characterize}, which describes the observed change in length normal to sheared faces of a cuboid and its effect on the axial length of a cylinder subjected to torsion \cite{poynting1909pressure}. The close connection of this effect to the material behaviour under shear stresses arises from considering infinitesimal cubes on the surface, visualized in Figure \ref{fig:simpleShearDeformation}. A positive Poynting effect, i.e.\ an increase in axial length corresponding to a value $b>1$, is is exhibited by most physical materials, although a contraction in axial direction (negative Poynting effect, $b<1$) has been observed for certain biopolymers \cite{mihai2011positive}.

Another effect connected to shear stresses is the so-called \emph{Kelvin effect} \cite[p.~176]{truesdell65} which describes a change in the volume, i.e.\ $\det F\neq 1$, of the material under pure shear stress. Note that this behaviour cannot be described by a hyperelastic material with an additive isochoric-volumetric split, i.e.\ an energy potential of the form $W(F)=\Wiso(F/(\det F)^\frac{1}{3})+f(\det F)$, since in this case \cite{richter1948} 
\begin{align}
	0=\tr(\sigma)=f'(\det V) \qquad\implies\qquad \det V = 1\label{eq:richter}
\end{align}
due to the usual requirement of a unique stress-free reference configuration.
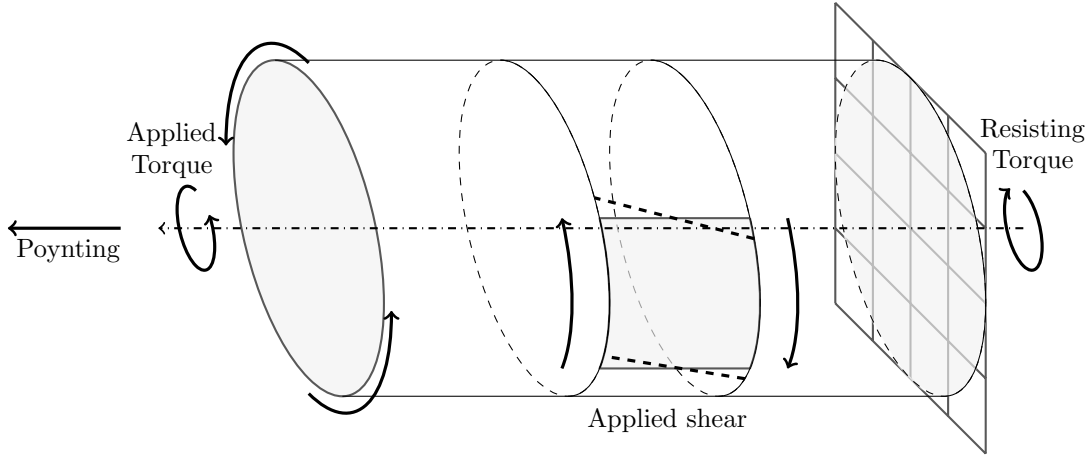
\begin{figure}[h!]
    \centering
%    \tikzRemake
    \tikzsetnextfilename{torsion}
    \begin{tikzpicture}
 		\input{tikz/tikz_torsion.tex}
    \end{tikzpicture}
    \caption{Infinitesimal planes tangential to the cylinder are subject to shear, and the generalized \emph{positive} Poynting effect for these planes leads to an elongation in the direction of the torsional axis (positive Poynting effect).}	\label{fig:torsion}
\end{figure}
We summarize:\par
\fbox{\parbox{.98\textwidth}{
	\[
		\text{Cauchy pure shear stress}\quad\sigmahat(B)=\begin{pmatrix}
			0&s&0\\s&0&0\\0&0&0
		\end{pmatrix}\quad\longrightarrow\quad B=\begin{pmatrix}
			B_{11}&B_{12}&0\\B_{21}&B_{22}&0\\0&0&B_{33}
		\end{pmatrix}
	\]
	$\begin{array}[t]{llll}
		\text{Positive Poynting effect:}&B_{11}>1\,,\qquad& \text{Negative Poynting effect:}&B_{11}<1\,.\\
		\text{Kelvin effect:}&\det B\neq 1\,,\qquad& \text{\enquote{not-planar}:}&B_{33}\neq 1\,.
	\end{array}$
}}\par
%
%
%
%%%%%%%%%%
\section{Idealized finite simple shear deformations}\label{section:finiteSimpleShear}
Of course, while any deformation gradient $F\in\GLp(3)$ corresponding to a Cauchy pure shear stress must be of the form \eqref{prop_F_eindeutig_bestimmt_B} regardless of the (isotropic) constitutive law of elasticity, the value of the parameters $a,b,c,\gamma$ or, equivalently, the principal stretches $\lambda_1,\lambda_2,\lambda_3$, depend on the specific stress-strain relation.

In particular, we believe that not every deformation of the general form \eqref{eq:general_form} is suitable to be called \enquote{shear}, motivating the following definition.\par
\fbox{\parbox{.98\textwidth}{
\begin{definition}\label{def:shearDeformation}
	We call $F=VR\in\GLp(3)$ with $V\in\PSym(3)$ and $R\in \SO(3)$ an (idealized) \emph{(finite) shear deformation} if the following requirements are satisfied:
	\begin{itemize}
		\item[i)] The stretch $V$ (or, equivalently, the deformation $F$) is \emph{volume preserving}, i.e.\ $\det V = 1$.
		\item[ii)] The stretch $V$ is \emph{planar}, i.e.\ $V$ has the eigenvalue $1$ to the eigenvector $e_3$.
		\item[iii)] The rotation $R$ is such that the deformation $F$ is \emph{ground parallel}, i.e.\ $e_1$, $e_3$ are eigenvectors of $F$.
	\end{itemize}
	In terms of the singular values $\lambda_1,\lambda_2,\lambda_3\in\R_+$ of $F$, i.e.\ the eigenvalues of $V$, the first two conditions can be stated as $\lambda_1\.\lambda_2\.\lambda_3=1$ and $\lambda_3=1$, respectively. In particular, i) and ii) are satisfied if and only if there exists $\lambda\in\R_+$ with $\lambda_1=\lambda$, $\lambda_2=\frac 1\lambda$ and $\lambda_3=1$.
\end{definition}}}
\par
\begin{remark}
	Definition \ref{def:shearDeformation} is a direct generalization of the infinitesimal behavior: the infinitesimal (classical) simple shear is planar, ground parallel and infinitesimally volume preserving, cf.\ Definition \ref{def:linearShear}. The concept is visualized in Figure \ref{fig:Tabular}. 
\end{remark}
As stated previously, the actual relation between the values of the amount of shear stress $s$ and the corresponding deformation tensor $B$ depends on the constitutive elasticity law. We want to give an example of a stress-strain relation for which Cauchy pure shear stresses induce deformations which do not satisfy Definition \eqref{def:shearDeformation}.
\begin{example}
	Consider the elastic energy potential $W$ and the corresponding Cauchy stress response $\sigmahat$ with
	\begin{equation}
		W(F)=\frac{1}{2}\norm{F}^2-\log(\det F)=\frac{1}{2}\left( I_1-\log I_3\right)\,,\qquad \sigmahat(B)=\frac{1}{\sqrt{\det B}}\left[B-\id\right]\,.
	\end{equation}
	If $\sigma=s\.(e_1\otimes e_2+e_2\otimes e_1)$ is a Cauchy pure shear stress, then the left Cauchy-Green deformation tensor $B$ is of the form \eqref{form_von_P} and thus
	\begin{align*}
		&&&\begin{pmatrix} 0&s&0\\s&0&0\\0&0&0	\end{pmatrix}=\frac{1}{\sqrt{I_3}}\begin{pmatrix}p-1&q&0\\q&p-1&0\\0&0&r-1\end{pmatrix}\qquad\iff\qquad r=1\,,\qquad p=1\quad \text{and}\quad s=\frac{q}{\sqrt{I_3}}\\
		&\implies& q&=s\.\sqrt{(p^2-q^2)\.r}\quad\overset{r=p=1}{\implies}\quad q^2=s^2\.(1-q^2)\quad\iff\quad\left(1+s^2\right)\.q^2=s^2\quad\implies\quad q=\frac{s}{\sqrt{1+s^2}}\,.
	\end{align*}
	Thus the deformation is planar ($r=1$) but not volume preserving for $s\neq0$, since
	\begin{equation*}
		\det B=(p^2-q^2)\.r=1-\frac{s^2}{1+s^2}=\frac{1}{1+s^2}<1\,.\null\tag*{\qed}
	\end{equation*}
\end{example}
We will demonstrate that the deformations of the general triaxial form \eqref{eq:general_form} which additionally satisfy the requirements of Definition \ref{def:shearDeformation} are exactly the left finite simple shear deformations introduced in Definition \ref{def:finiteSimpleShear}. Similarly, as shown in Appendix \ref{sec:Biot}, any shear deformation that corresponds to a Biot pure shear stress is a right finite simple shear deformation.
\begin{lemma}\label{lem_form_von_p_kommutieren_mit_t}
	Let $T$ be a pure shear stress tensor with $s\neq 0$ and $P\in\PSym(3)$ such that the eigenvalues $\lambda_1^2,\lambda_2^2,\lambda_3^2$ of $P$ are given by $\lambda_1=\lambda$, $\lambda_2=\frac 1\lambda$ and $\lambda_3=1$ for arbitrary $\lambda\in\R_+$.
	Then $P$ and $T$ commute if and only if $P$ is of the form
	\begin{equation}
		\sqrt P = \matr{\cosh(\alpha)&\sinh(\alpha)&0\\\sinh(\alpha)&\cosh(\alpha)&0\\0&0&1}=\exp\matr{0&\alpha&0\\\alpha&0&0\\0&0&0}\qquad\text{with $\alpha\in\R$}\label{form_von_P_shear}
	\end{equation}
	i.e.\ if and only if $\sqrt P$ is a finite pure shear stretch $V_\alpha$ of the form \eqref{eq:definitionPureShearStretch}.
\end{lemma}
\begin{proof}
	By \reflemma{lem_P_und_T_kommutieren}, the tensor $P$ commutes with $T$ if and only if it is of the form \eqref{eq:PinEigenvalues}. 
	Let $\alpha\colonequals\log(\lambda)\in\R$. Then, for $\lambda_1=\lambda$, $\lambda_2=\frac 1\lambda$ and $\lambda_3=1$, eq.\ \eqref{eq:squarerootGeneralForm} yields
	\begin{align*}
		\sqrt P &=  \frac 12\begin{pmatrix} \lambda_1 + \lambda_2 & \lambda_1 - \lambda_2 & 0 \\ \lambda_1 - \lambda_2 & \lambda_1 + \lambda_2 & 0 \\ 0 & 0 & 2\lambda_3 \end{pmatrix}= \frac 12\begin{pmatrix} e^\alpha + e^{-\alpha} & e^\alpha - e^{-\alpha} & 0 \\ e^\alpha - e^{-\alpha} & e^\alpha + e^{-\alpha} & 0 \\ 0 & 0 & 2 \end{pmatrix} = \begin{pmatrix} \cosh(\alpha) & \sinh(\alpha) & 0 \\ \sinh(\alpha) & \cosh(\alpha) & 0 \\ 0 & 0 & 1 \end{pmatrix}\,.\qedhere
	\end{align*}
\end{proof}
\begin{remark}
	In contrast to infinitesimal pure shear strain (see Definition \ref{def:linearShear} and Remark \ref{remark:linearStretchNotFinitelyVolumePreserving}), the finite pure shear stretch is finitely volume preserving as well (since $\det V_\alpha=1$). This fact is visualized for selected $\gamma$ and $\alpha$ in Figure \ref{fig_linear_finite_pureShearStretch}: The determinant of infinitesimal pure shear stretch is $\det(\id+\eps_\gamma)=1-\frac{\gamma^2}4$, whereas the determinant of finite pure shear stretch is $1$. Note that finite pure shear stretch linearizes to linear pure shear stretch, cf. Appendix \ref{appendix:LinearShear}.
	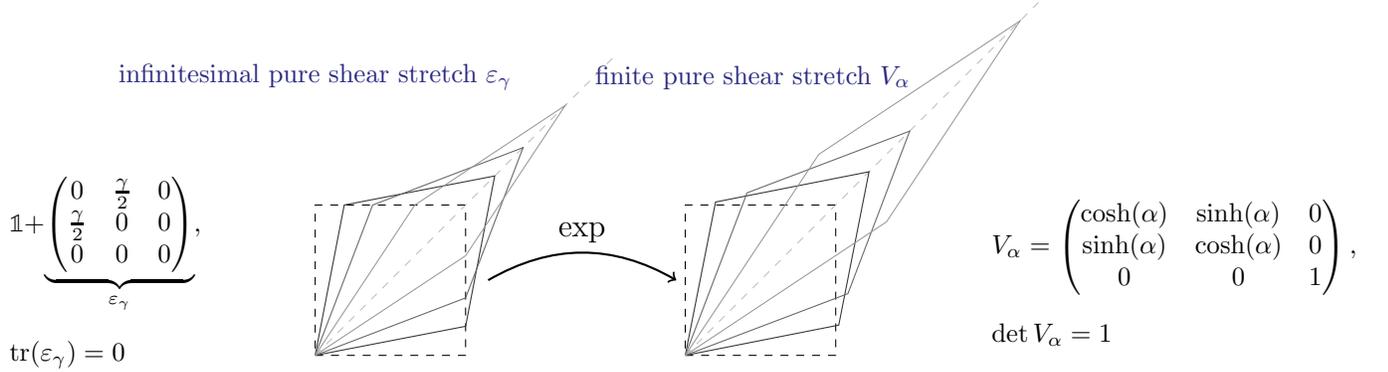
\begin{figure}[h!]
%    	\tikzRemake
    	\tikzsetnextfilename{pureShearStretchLinearFinite}
    	\hspace{1em}
    	\begin{minipage}{0.12\textwidth}
    		\vspace{6em}
     		\[
       			\id+\underbrace{\matr{0 & \frac\gamma2 & 0 \\ \frac\gamma2 & 0 & 0 \\ 0 & 0 & 0}}_{\eps_\gamma},
      		\]
      		$\tr(\eps_\gamma)=0$
    	\end{minipage}
    	\hspace{-1cm}
    	\begin{minipage}{0.6\textwidth}
      		\begin{tikzpicture}
        		\input{tikz/tikz_pureShearStretchLinearFinite.tex}
      		\end{tikzpicture}
    	\end{minipage}
    	\hspace{1cm}
   	 	\begin{minipage}{0.15\textwidth}
    		\vspace{6em}
      		\[
        		V_\alpha=\matr{\cosh(\alpha) & \sinh(\alpha) & 0 \\ \sinh(\alpha) & \cosh(\alpha) & 0 \\ 0 & 0 & 1},
      		\]
      		$\det V_\alpha=1$
    	\end{minipage}
    	\caption{Linear and finite pure shear stretch with $\gamma=2\.\alpha$. Both are infinitesimally volume preserving, but only the finite pure shear stretch leaves the volume unchanged for finite stretch ratios. The transition mechanism is the matrix exponential.}\label{fig_linear_finite_pureShearStretch}
	\end{figure} 
\end{remark}
\begin{lemma}\label{lem_F_simple_finite_shear_B}
	Let $B=P$ be given by \eqref{form_von_P_shear}. Then any deformation gradient $F \in \GLp(3)$ with $FF^T=B$ is of the form
	\begin{equation}
		F = \begin{pmatrix} \frac 1{\sqrt{\cosh(2\alpha)}} & \frac{\sinh(2\alpha)}{\sqrt{\cosh(2\alpha)}} & 0 \\ 0 & \sqrt{\cosh(2\alpha)} & 0 \\ 0 & 0 & 1\end{pmatrix}\cdot Q = \frac 1{\sqrt{\cosh(2\alpha)}}\begin{pmatrix} 1 & \sinh(2\alpha) & 0 \\ 0 & \cosh(2\alpha) & 0 \\ 0 & 0 & \sqrt{\cosh(2\alpha)}\end{pmatrix}\cdot Q
	\end{equation}
	with arbitrary $Q\in\SO(3)$.
\end{lemma}
\begin{proof}
	We use \refproposition{prop_F_eindeutig_bestimmt_B} and write $\lambda_1=e^\alpha$, $\lambda_2=e^{-\alpha}$, $\lambda_3=1$. Then
	\begin{align*}
		\gamma &=  \frac{\lambda_1^2-\lambda_2^2}{\lambda_1^2+\lambda_2^2}  = \frac{e^{2\alpha}-e^{-2\alpha}}{e^{2\alpha}+e^{-2\alpha}} = \tanh(2\alpha)\,,\qquad b = \sqrt{\frac{\lambda_1^2+\lambda_2^2}2}  = \sqrt{\frac{e^{2\alpha} + e^{-2\alpha}}2} = \sqrt{\cosh(2\alpha)}\,,\\
		 a &=  \lambda_1\lambda_2\sqrt{\frac 2{\lambda_1^2+\lambda_2^2}}  = e^{\alpha}e^{-\alpha}\cdot\sqrt{\frac 2{e^{2\alpha} + e^{-2\alpha}}} = \frac 1{\sqrt{\cosh(2\alpha)}}\qquad\text{and}\qquad	c = \lambda_3  = 1
	\end{align*}
	and thus
	\begin{align*}
		F &=  \begin{pmatrix} 1 & \gamma & 0 \\ 0 & 1 & 0 \\ 0 & 0 & 1 \end{pmatrix} \begin{pmatrix} a & 0 & 0 \\ 0 & b & 0 \\ 0 & 0 & c \end{pmatrix}Q= \begin{pmatrix} 1 & \tanh(2\alpha) & 0 \\ 0 & 1 & 0 \\ 0 & 0 & 1 \end{pmatrix} \begin{pmatrix} \frac 1{\sqrt{\cosh(2\alpha)}} & 0 & 0 \\ 0 & \sqrt{\cosh(2\alpha)} & 0 \\ 0 & 0 & 1 \end{pmatrix}Q\\
		&= \frac 1{\sqrt{\cosh(2\alpha)}}\begin{pmatrix} 1 & \sinh(2\alpha) & 0 \\ 0 & \cosh(2\alpha) & 0 \\ 0 & 0 & \sqrt{\cosh(2\alpha)} \end{pmatrix}Q\,.\qedhere
	\end{align*}
\end{proof}
\begin{figure}[h!]
    \centering
%    \tikzRemake
    \tikzsetnextfilename{PolarDecompositionFiniteSimpleShear1}
    \begin{tikzpicture}
        \input{tikz/tikz_PolarDecompositionFiniteSimpleShear1.tex}
    \end{tikzpicture}
    \caption{Left finite simple shear deformation with amount of shear $\gamma=\tanh(2\alpha)$, cf.\ \cite{lubensky2005}.}\label{fig:polarDexompositionLeft}
\end{figure}
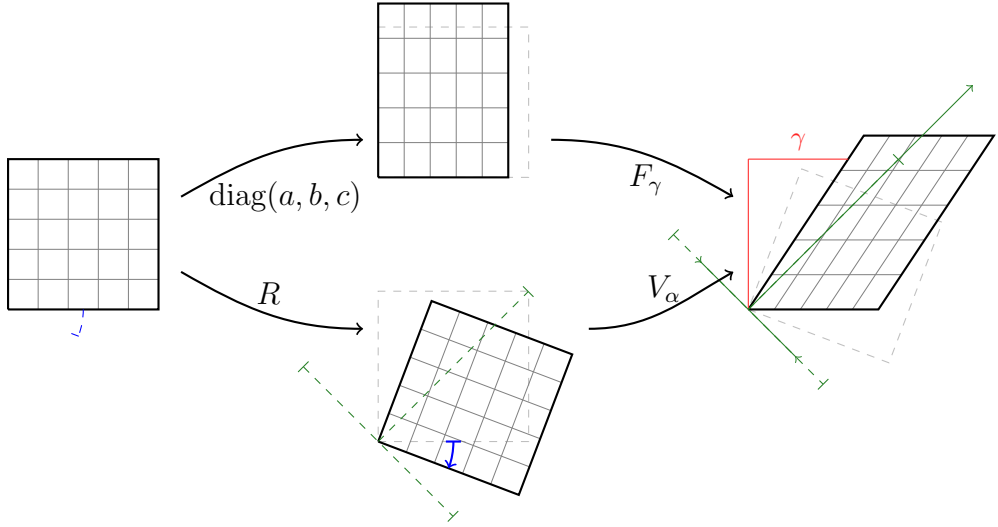 
In particular, the triaxial stretch $\diag(a,b,c)$ has to be a (diagonal) biaxial pure shear stretch of the form $\diag(a,\frac{1}{a},1)$ if the deformation is a finite shear in the sense of Definition \ref{def:linearShear}. The exact form the deformation can take in this case is given by the following theorem.
\begin{theorem}\label{theorem:main}
	Any \emph{idealized shear deformation}\footnote{Again, that is: $\det F = 1$, $e_3$ is an eigenvector of $F$ for eigenvalue $1$ and $e_1$ is an eigenvector of $F$.\label{footnote}} $F\in\GLp(3)$ satisfying Definition \ref{def:shearDeformation} that corresponds to a (non-trivial) Cauchy pure shear stress for an isotropic law of elasticity is a \emph{left finite simple shear deformation} of the form
	\begin{equation}\label{eq:leftSimpleFiniteShearTheorem}
		F_\alpha = \frac 1{\sqrt{\cosh(2\alpha)}}\matr{ 1 & \sinh(2\alpha) & 0 \\ 0 & \cosh(2\alpha) & 0 \\ 0 & 0 & \sqrt{\cosh(2\alpha)} }\qquad\text{with $\alpha\in\R$}\,.
	\end{equation}
\end{theorem}
\begin{proof}
	Let $\sigmahat$ be a pure shear stress tensor with $s\neq 0$ and $B\in\PSym(3)$ with $\lambda_1=\lambda$, $\lambda_2=\frac 1\lambda$, $\lambda_3=1$ for arbitrary $\lambda\in\R_+$ with $\lambda_1^2,$ $\lambda_2^2,$ $\lambda_3^2$ as the eigenvalues of $B$. By \reflemma{lem_form_von_p_kommutieren_mit_t}, the tensors $B$ and $\sigmahat(B)$ commute if and only if $B=V^2$ is of the form \eqref{form_von_P_shear}, i.e.\ if and only if
	\begin{align}
		V = V_\alpha = \matr{\cosh(\alpha)&\sinh(\alpha)&0\\\sinh(\alpha)&\cosh(\alpha)&0\\0&0&1}=\exp\matr{0&\alpha&0\\\alpha&0&0\\0&0&0}\qquad\text{with $\alpha\in\R$}
	\end{align}
	is a finite pure shear stretch.

	For an isotropic elasticity law the left Cauchy-Green deformation tensor $B=FF^T$ commutes with $\widehat\sigma(B)$.
	Thus if $\widehat\sigma(B)$ is a Cauchy pure shear stress, then $B=FF^T$ is of the form \eqref{form_von_P_shear} (i.e.\ $V=V_\alpha$ is a finite pure shear stress). By \reflemma{lem_F_simple_finite_shear_B}, the deformation gradient $F$ is determined up to an arbitrary right-hand side rotation $Q\in\SO(3)$, and it remains to observe that exactly for $Q=\id$, the resulting deformation $F_\alpha$ satisfies the condition of ground parallelism (i.e.\ has eigenvectors $e_1$ and $e_3$).
\end{proof}
Observe that $b=\sqrt{\cosh(2\alpha)}=\sqrt{B_{11}}\geq 1$, hence a finite simple shear deformation $F_\alpha$ always results in a positive Poynting effect, which is considered a physically plausible behavior for most materials. 

In Appendix \ref{sec:Biot} we prove an analogous result for Biot pure shear stress and right finite shear deformations. The following connection between the left and right finite simple shear deformation is visualized in Figure \ref{fig:polarDexompositionLeft}.
\begin{lemma}\label{lemma:polarFiniteShear}
	Let $\alpha\in\R$ and
	\begin{equation}
	\label{eq:polarFiniteShearFactors}
		V_\alpha = \begin{pmatrix} \cosh(\alpha) & \sinh(\alpha) & 0 \\ \sinh(\alpha) & \cosh(\alpha) & 0 \\ 0 & 0 & 1 \end{pmatrix}
		\,,\qquad
		R = \frac 1{\sqrt{\cosh(2\alpha)}}\begin{pmatrix} \cosh(\alpha) & \sinh(\alpha) & 0 \\ -\sinh(\alpha) & \cosh(\alpha) & 0 \\ 0 & 0 & \sqrt{\cosh(2\alpha)} \end{pmatrix}\,.
	\end{equation}
	Then
	\begin{equation}
		V_\alpha\.R = \frac 1{\sqrt{\cosh(2\alpha)}}\begin{pmatrix} 1 & \sinh(2\alpha) & 0 \\ 0 & \cosh(2\alpha) & 0 \\ 0 & 0 & \sqrt{\cosh(2\alpha)}\end{pmatrix}\,,\qquad R\.V_\alpha = \frac 1{\sqrt{\cosh(2\alpha)}}\begin{pmatrix} \cosh(2\alpha) & \sinh(2\alpha) & 0 \\ 0 & 1 & 0 \\ 0 & 0 & \sqrt{\cosh(2\alpha)}\end{pmatrix}
		\end{equation}
	are a \emph{left finite simple shear deformation} and a \emph{right finite simple shear deformation}, respectively.
\end{lemma}
\begin{proof}
	The proof follows from direct computation, using the general equalities
	\[
		\cosh(\alpha)^2 + \sinh(\alpha)^2 = \cosh(2\alpha)\qquad\text{and}\qquad 2\.\sinh(\alpha)\.\cosh(\alpha) = \sinh(2\alpha)\,.\qedhere
	\]
\end{proof}
\begin{remark}
	Lemma \ref{lemma:polarFiniteShear} states that for a given $ \alpha \in \R $, the left and the right finite simple shear deformation are generated by exactly the same stretch $V_\alpha$ and rotation $R$, the only difference being the order of application of the two linear operations.
\end{remark}
\begin{figure}[h!]
    \centering
    \tikzRemake
    \tikzsetnextfilename{shearTabular}
    
\fbox{\parbox{1.00\textwidth}{
    \scalebox{0.825}{
    \begin{tikzpicture}
        \input{tikz/tikz_shearTabular.tex}
    \end{tikzpicture}}}}
    \caption{Comparison of pure shear stress and finite simple shear.}\label{fig:Tabular}
\end{figure}
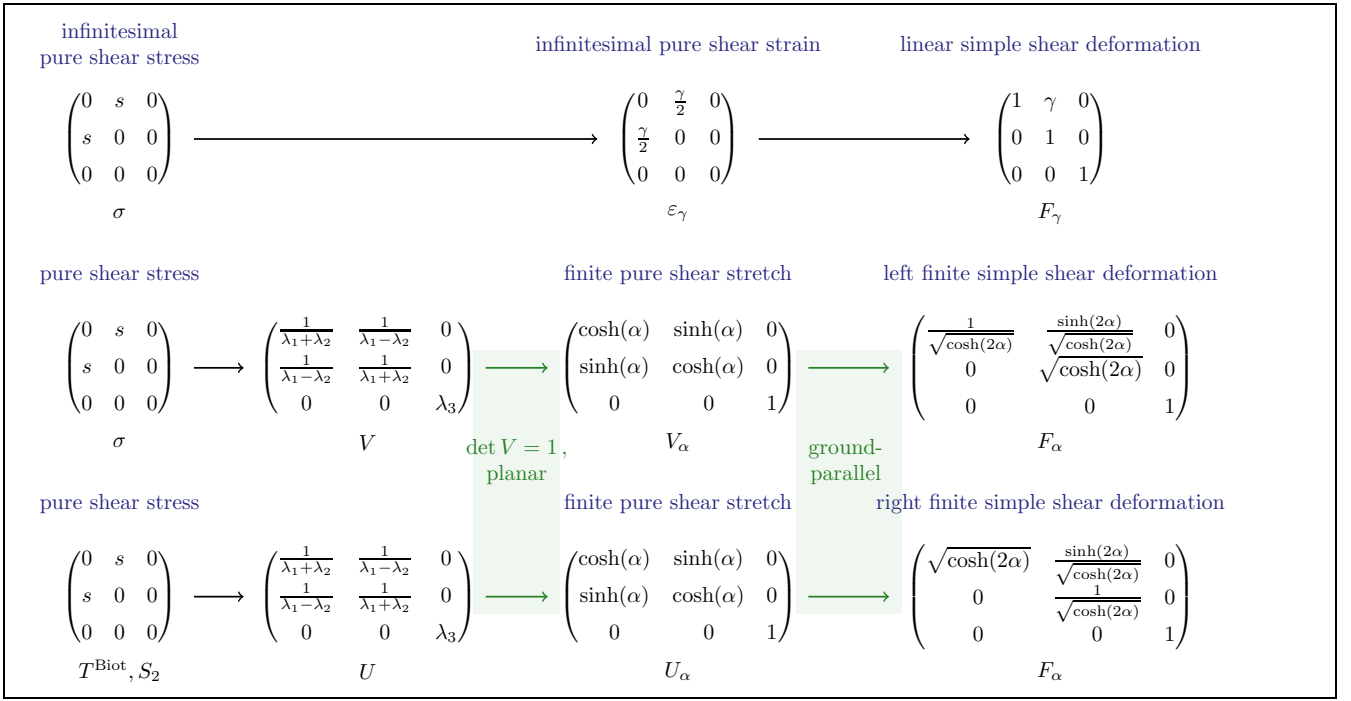 
%
%
%
%
%%%%%%%%%%
\section{Constitutive conditions for idealized shear response in Cauchy pure shear stress}\label{sec:Constitutive}
As stated in Corollary \ref{corollary:destradeWir}, a Cauchy pure shear stress corresponds with a deformation gradient $F$ of the general triaxial form \eqref{eq:general_form}. Theorem \ref{theorem:main} shows that if $F_\alpha$ also satisfies Definition \ref{def:shearDeformation}, then $F$ must be of the form \eqref{eq:leftSimpleFiniteShearTheorem}. However, it is important to note that whether or not a deformation gradient $F$ corresponding to a pure shear stress satisfies Definition \ref{def:shearDeformation} depends on the particular stress response function. Similarly, not every constitutive law ensures that every idealized finite shear of the form \eqref{eq:leftSimpleFiniteShearTheorem} induces a pure shear stress tensor.

On the other hand, a number of reasons suggest that for an \emph{idealized} elastic material, finite simple shear deformations \emph{should} correspond to pure shear Cauchy stresses; for example:
\begin{itemize}
	\item Cauchy pure shear stress does not induce a pressure load, i.e.\ $\tr(\sigma)=0$. Therefore, it should be assumed that the pure stretch response is \emph{volume-preserving} (no Kelvin effect).\footnote{Moon and Truesdell \cite[2]{moon1974interpretation} remark that \enquote{[i]t is easily possible to conceive a \emph{truly isotropic solid} that [\ldots] retains its volume unchanged when subjected to a shear stress of any amount}; in the following, we give examples of such material models.} This corresponds to the case of infinitesimal shear, where the linear strain tensor $\eps$ is trace free.
	\item Cauchy pure shear stress acts only in the $e_1$-$e_2$-plane. Therefore, it is plausible (for an \emph{idealized} shear) to assume a pure stretch response only in the corresponding plane and no deformation in the third component ($\lambda_3=1$). Again, this behavior is always exhibited in the linearized case.
\end{itemize}

In the following, we therefore examine constitutive conditions for the elasticity law which ensure that \emph{every} left finite simple shear deformation $F_\alpha$ is mapped to a Cauchy pure shear stress. In Section \ref{section:shearCompatibilityReverse}, we will also discuss additional requirements which ensure the converse implication, i.e.\ conditions under which every Cauchy stress always induces a pure shear stretch. Again, note carefully that those two implications are \emph{distinct} properties of an elastic response function (cf.\ \cite{lurie2012nonlinear}), since the mapping $V\mapsto\sigma(V)=\sigmahat(V^2)$ is generally not invertible.
%
%
%%%%%%%%%%
\subsection{Pure shear stress induced by pure shear stretch}
Recall from Lemma \ref{lem_P_und_T_kommutieren} that for any left finite simple shear deformation, all eigenvectors of $B=F_\alpha {F_\alpha}^T$ or $V_\alpha$ are eigenvectors of a Cauchy pure shear stress $\sigmahat(FF^T)$. More specifically,
\begin{align}
	V_\alpha=\sqrt{F_\alpha F_\alpha^T} =  \matr{\cosh(\alpha)&\sinh(\alpha)&0\\ \sinh(\alpha)&\cosh(\alpha)&0\\ 0&0&1} = Q\,\cdot\underbrace{\dmatr{\lambda}{\frac{1}{\lambda}}{1}}_{\mathclap{\text{pure shear deformation}}}\,\cdot\,\,Q^T\quad\text{ with }\quad \lambda=e^\alpha\label{eq:revision1}
\end{align}
and
\begin{align}
	\widehat\sigma(B)=\matr{0&s&0\\s&0&0\\0&0&0} = Q\,\dmatr{s}{-s}{0}\,Q^T\quad\text{ with }\quad Q = \frac{\sqrt{2}}{2} \, \matr{1&-1&0 \\ 1&1&0 \\ 0&0&1} \in \SO(3)\,.\label{eq:revision2}
\end{align}
Therefore, whether the Cauchy stress $\widehat\sigma(FF^T)$ corresponding to a left finite simple shear (or, equivalently, the Cauchy stress $\widehat\sigma(V^2)$ corresponding to a finite pure shear stretch $V=V_\alpha$ of the form \eqref{eq:definitionPureShearStretch}) is a pure shear stress depends only on the principal stresses $\sigma_1,\sigma_2,\sigma_3$. In particular, by comparing the diagonal matrices in \eqref{eq:revision1} and \eqref{eq:revision2}, we find that for a given stress response, a finite pure shear stretch \emph{always} induces a pure shear stress \emph{if and only if} for all $\lambda\in\Rp$, there exists $s\in\R$ such that
\begin{align}
	B=V^2= \dmatr{\lambda^2}{\frac{1}{\lambda^2}}{1} \qquad\implies\qquad \widehat\sigma(B) = \dmatr{s}{-s}{0}\label{eq:shearStretchStress}
\end{align}
or, equivalently,
\[
	\lambda_1=\lambda\,,\quad \lambda_2=\frac1\lambda\,,\quad \lambda_3=1 \qquad\implies\qquad
	\sigma_1=s\,,\quad \sigma_2=-s\,,\quad \sigma_3=0\,,
\]
where $\sigma_i$ denotes the $i$--th eigenvalue of $\widehat\sigma(FF^T)$ for $F=\diag(\lambda_1,\lambda_2,\lambda_3)$.

Recall that any isotropic Cauchy stress response $B\mapsto\sigmahat(B)$ can be expressed in the form \cite{richter1948}
\begin{align}
\label{eq:CauchyInBetas}
	\sigma=\beta_0\.\id+\beta_1\.B+\beta_{-1}\.B^{-1}
\end{align}	
with scalar-valued functions $\beta_0$, $\beta_1$, $\beta_{-1}$ depending on the matrix invariants $I_1$, $I_2$, $I_3$ of the stretch tensor $B$. In terms of this representation, equality \eqref{eq:shearStretchStress} reads
\begin{equation}
\label{eq:shearCorrespondenceInvariantsMatrix}
	\dmatr{s}{-s}{0}=\beta_0\dmatr{1}{1}{1}+\beta_1\dmatr{\lambda^2}{\frac{1}{\lambda^2}}{1}+\beta_{-1}\dmatr{\frac{1}{\lambda^2}}{\lambda^2}{1}\,.
\end{equation}
Assume without loss of generality that $\lambda\neq1$ and let $s=\beta_0+\beta_1\.\lambda^2+\beta_{-1}\.\frac{1}{\lambda^2}$. Then \eqref{eq:shearCorrespondenceInvariantsMatrix} is equivalent to the equations
\begin{equation}
\label{eq:shearCorrespondenceInvariantsScalars}
	-\left(\beta_0+\beta_1\.\lambda^2+\beta_{-1}\.\frac{1}{\lambda^2}\right)=\beta_0+\beta_1\.\frac{1}{\lambda^2}+\beta_{-1}\.\lambda^2\quad\text{ and }\quad	0=\beta_0+\beta_1+\beta_{-1}\,.
\end{equation}
Using the second equation, we can simplify the first one to
\begin{align*}
	0=2\.\beta_0+\beta_1\left(\lambda^2+\frac{1}{\lambda^2}\right)+\beta_{-1}\left(\frac{1}{\lambda^2}+\lambda^2\right)=2\underbrace{\left(\beta_0+\beta_1+\beta_{-1}\right)}_{=0}+\underbrace{\left(\lambda^2+\lambda^{-2}-2\right)}_{> 0}\left(\beta_1+\beta_{-1}\right)\,,
\end{align*}
i.e.\ to $\beta_1+\beta_{-1}=0$, which reduces the second equation $\eqref{eq:shearCorrespondenceInvariantsScalars}_2$ to $\beta_0=0$. The condition \eqref{eq:shearCorrespondenceInvariantsScalars} can therefore be stated as\par
\fbox{\parbox{.99\textwidth}{
\setlength{\abovedisplayskip}{0pt}
\setlength{\belowdisplayskip}{0pt}
\begin{align}
	\beta_1+\beta_{-1}=0\quad\text{ and }\quad\beta_0=0\qquad\text{for all }\lambda\in\Rp\text{ with }\lambda_1=\lambda\,,\; \lambda_2=\frac1\lambda\,,\; \lambda_3=1\,.\label{eq:CompatibilityBetas}
\end{align}}}\par
In the following, we will assume that the elasticity law is \emph{hyperelastic}, i.e.\ induced by an elastic energy function $W\col\GLp(3)\to\R$. Then the functions $\beta_i$ can be expressed by
\begin{align}
	\beta_0=\frac{2}{\sqrt{I_3}}\left(I_2\.\frac{\partial W}{\partial I_2}+I_3\.\frac{\partial W}{\partial I_3}\right)\,,\qquad\beta_1=\frac{2}{\sqrt{I_3}}\.\frac{\partial W}{\partial I_1}\,,\qquad\beta_{-1}=-2\sqrt{I_3}\.\frac{\partial W}{\partial I_2}\label{eq:CauchyHyper}
\end{align}
with the matrix invariants
\begin{align}
	I_1=\tr B=\lambda_1^2+\lambda_2^2+\lambda_3^2\,,\qquad I_2=\tr(\Cof B)=\lambda_1^2\.\lambda_2^2+\lambda_1^2\.\lambda_3^2+\lambda_2^2\.\lambda_3^2\,,\qquad I_3=\det B=\lambda_1^2\.\lambda_2^2\.\lambda_3^2\,.
\end{align}
of $B$. If the singular values of $F$ are given by $\lambda_1=\lambda$, $\lambda_2=\frac1\lambda$ and $\lambda_3=1$, then $I_1=I_2=1+\lambda^2+\frac{1}{\lambda^2}\geq 3$ and $I_3=1$. The compatibility conditions \eqref{eq:CompatibilityBetas} can be restated in terms of the hyperelastic energy function:
\begin{alignat*}{5}
	&&\beta_1+\beta_{-1}&=0\quad&&\text{ and }&\quad\beta_0&=0\qquad&&\text{for all }I_1=I_2\geq 3\,,\;I_3=1\\
	\iff\;\;&& \frac{2}{\sqrt{I_3}}\.\frac{\partial W}{\partial I_1}&=2\sqrt{I_3}\.\frac{\partial W}{\partial I_2}\quad&&\text{ and }&\quad\frac{2}{\sqrt{I_3}}\left(I_2\.\frac{\partial W}{\partial I_2}+I_3\.\frac{\partial W}{\partial I_3}\right)&=0\qquad&&\text{for all }I_1=I_2\geq 3\,,\;I_3=1
\end{alignat*}
or, equivalently,\par
\fbox{\parbox{.99\textwidth}{
\setlength{\abovedisplayskip}{0pt}
\setlength{\belowdisplayskip}{0pt}
\begin{align}
	\frac{\partial W}{\partial I_1}&=\frac{\partial W}{\partial I_2}&\quad&\text{ and }&\quad I_2\.\frac{\partial W}{\partial I_2}+\frac{\partial W}{\partial I_3}&=0\qquad\text{for all }I_1=I_2\geq 3\,,\;I_3=1\,.\label{eq:CompatibilityConditions}
\end{align}}}\par

Many isotropic energy functions are expressed more conveniently in terms of the principal stretches instead of the three matrix invariants $I_1$, $I_2$ and $I_3$, i.e.\ in the form
\begin{equation}
\label{eq:energySingularValueRepresentation}
	W(F) = W(\lambda_1,\lambda_2,\lambda_3) \qquad\text{ for all }\; F\in\GLp(3) \;\text{ with singular values }\; \lambda_1,\lambda_2,\lambda_3\,.
\end{equation}
The isotropy implies that this representation is invariant under permutations of the arguments. In order to find conditions on $W$ which ensure that our finite pure shear stretches correspond to pure shear stresses, we use the general formula \cite[p.216]{Ogden83}
\begin{equation}
\label{eq:ogdenFormula}
	\sigma_i = \frac{\lambda_i}{\lambda_1\,\lambda_2\,\lambda_3}\.\pdd{W}{\lambda_i}(\lambda_1,\lambda_2,\lambda_2)\,.
\end{equation}
for the eigenvalues of $\sigmahat(B)$.
Again, we want to ensure that \eqref{eq:shearStretchStress} holds.
Similar to \eqref{eq:CompatibilityBetas}, this leads to the two equations
\begin{align}
	\sigma_1+\sigma_2=0\quad\text{ and }\quad\sigma_3=0\qquad\text{for all }\lambda\in\Rp\text{ with }\lambda_1=\lambda\,,\; \lambda_2=\frac1\lambda\,,\; \lambda_3=1\,.\label{eq:CompatibilitySigmas}
\end{align}
With \eqref{eq:ogdenFormula}, these two compatibility conditions read\par
\fbox{\parbox{0.99\textwidth}{
\setlength{\abovedisplayskip}{0pt}
\setlength{\belowdisplayskip}{0pt}
\begin{equation}
	\lambda\.\frac{\partial W}{\partial \lambda_1}+\frac{1}{\lambda}\.\frac{\partial W}{\partial \lambda_2}=0\quad\text{ and }\quad\frac{\partial W}{\partial \lambda_3}=0\qquad\text{for all }\lambda\in\Rp\text{ with }\lambda_1=\lambda\,,\; \lambda_2=\frac1\lambda\,,\; \lambda_3=1\,.\label{eq:CompatibilityConditions2}
\end{equation}}}

The following Lemma gives a simple sufficient criterion for eq.\ \eqref{eq:CompatibilityConditions2} and thus for the desired implication \eqref{eq:shearStretchStress}.
	
\begin{lemma}
\label{lemma:tensionCompressionSymmetric}
	Let $W\col\GLp(3)\to\R$ be a sufficiently smooth isotropic elastic energy potential. If $W$ is \emph{tension-compression symmetric}, i.e.\ if $W(F\inv)=W(F)$ for all $F\in\GLp(3)$, then $\sigmahat(B)=\sigmahat(V^2)$ is a pure shear stress for every finite pure shear stretch $V=V_\alpha$.
\end{lemma}
\begin{proof}
	In terms of the singular values, the condition of tension-compression symmetry reads
	\begin{equation}\label{eq:tensionCompressionSymmetrySingularValues}
		W\left(\frac{1}{\lambda_1},\frac{1}{\lambda_2},\frac{1}{\lambda_3}\right) = W(\lambda_1,\lambda_2,\lambda_3) \qquad\text{for all }\;\lambda_1,\lambda_2,\lambda_2\in\Rp\,,
	\end{equation}
	which yields
	\begin{align*}
		\frac{\partial W}{\partial \lambda_1}\left(\lambda,\frac{1}{\lambda},1\right)&=\ddt\left. W\left(\lambda+t,\frac{1}{\lambda},1\right)\right|_{t=0}=\ddt\left. W\left(\frac{1}{\lambda+t},\lambda,1\right)\right|_{t=0}=\ddt\left. W\left(\lambda,\frac{1}{\lambda+t},1\right)\right|_{t=0}\\
		&=\left.\frac{\partial W}{\partial \lambda_2}\left(\lambda,\frac{1}{\lambda+t},1\right)\frac{-1}{(\lambda+t)^2}\right|_{t=0}=-\frac{1}{\lambda^2}\frac{\partial W}{\partial \lambda_2}\left(\lambda,\frac{1}{\lambda},1\right)
	\end{align*}
	and
	\begin{align*}
		\frac{\partial W}{\partial \lambda_3}\left(\lambda,\frac{1}{\lambda},1\right)&=\ddt\left.W\left(\lambda,\frac{1}{\lambda},1+t\right)\right|_{t=0}=\ddt\left.W\left(\frac{1}{\lambda},\lambda,\frac{1}{1+t}\right)\right|_{t=0}=\ddt\left.W\left(\lambda,\frac{1}{\lambda},\frac{1}{1+t}\right)\right|_{t=0}\\
		&=\left.\frac{\partial W}{\partial \lambda_3}\left(\lambda,\frac{1}{\lambda},\frac{1}{1+t}\right)\frac{-1}{(1+t)^2}\right|_{t=0}=-\frac{\partial W}{\partial \lambda_3}\left(\lambda,\frac{1}{\lambda},1\right)\,,
	\end{align*}
	thus the conditions \eqref{eq:CompatibilityConditions2} are both satisfied in this case.
\end{proof}

In terms of the principal invariants, the tension-compression symmetry of an energy $W$ can also be expressed as $W(I_1,I_2,I_3)=W(\frac{I_2}{I_3},\frac{I_1}{I_3},\frac{1}{I_3})$.
In the context of \emph{anti-plane shear deformations} (cf.\ Appendix \ref{appendix:shearMonotonicity}), it was recently shown \cite{voss2018aps} that the so-called \enquote{condition K1 with $b=\frac{1}{2}$}, which is exactly the equality $\frac{\partial W}{\partial I_1}=\frac{\partial W}{\partial I_2}$ given in eq.\ \eqref{eq:CompatibilityConditions}, is always satisfied by tension-compression symmetric energies.

Lemma \ref{lemma:tensionCompressionSymmetric} can be extended to include a volumetric term, yielding a more applicable sufficient criterion for the compatibility of shear deformations and shear stress.

\begin{theorem}
\label{theorem:shearCompatibilityMainResult}
	Let $W\col\GLp(3)\to\R$ be an elastic energy potential of the form
	\begin{equation}
		W(F) = \Wsym(F) + f(\det F)\,,
	\end{equation}
	where $\Wsym\col\GLp(3)\to\R$ is a sufficiently smooth tension-compression symmetric function and $f\col\Rp\to\R$ is differentiable with $f'(1)=0$. Then $\sigmahat(B)=\sigmahat(V^2)$ is a pure shear stress for every finite pure shear stretch $V=V_\alpha$.
\end{theorem}
\begin{proof}
	Since
	\[
		\frac{\partial W}{\partial \lambda_i}(\lambda_1,\lambda_2,\lambda_3)
		= \frac{\partial \Wsym}{\partial \lambda_i}(\lambda_1,\lambda_2,\lambda_3) + \frac{d}{d\lambda_i}\.f(\lambda_1\lambda_2\lambda_3)
		= \frac{\partial \Wsym}{\partial \lambda_i}(\lambda_1,\lambda_2,\lambda_3) + \frac{\lambda_1\lambda_2\lambda_3}{\lambda_i}\.f'(\lambda_1\lambda_2\lambda_3)\,,
	\]
	we can apply Lemma \ref{lemma:tensionCompressionSymmetric} to $\Wsym$ and compute
	\[
		\frac{\partial W}{\partial \lambda_1}\left(\lambda,\frac{1}{\lambda},1\right)
		= \frac{\partial \Wsym}{\partial \lambda_1}\left(\lambda,\frac{1}{\lambda},1\right) + \underbrace{ \frac{1}{\lambda}\.f'(1) \vphantom{\frac{1}{\frac1\lambda}} }_{=0}
		= -\frac{1}{\lambda^2}\frac{\partial \Wsym}{\partial \lambda_2}\left(\lambda,\frac{1}{\lambda},1\right) - \underbrace{\frac{1}{\frac1\lambda}\.f'(1)}_{=0}
		= -\frac{1}{\lambda^2}\frac{\partial W}{\partial \lambda_2}\left(\lambda,\frac{1}{\lambda},1\right)
	\]
	as well as
	\[
		\frac{\partial W}{\partial \lambda_3}\left(\lambda,\frac{1}{\lambda},1\right)
		= \frac{\partial \Wsym}{\partial \lambda_3}\left(\lambda,\frac{1}{\lambda},1\right) + \underbrace{f'(1)}_{=0}
		= -\frac{\partial \Wsym}{\partial \lambda_2}\left(\lambda,\frac{1}{\lambda},1\right) - \underbrace{f'(1)}_{=0}
		= -\frac{\partial W}{\partial \lambda_3}\left(\lambda,\frac{1}{\lambda},1\right)
	\]
	for all $F\in\GLp(3)$ with singular values $\lambda_1,\lambda_2,\lambda_3$, thus both conditions \eqref{eq:CompatibilityConditions2} are satisfied under the theorem's assumptions.
\end{proof}

\begin{example}\label{lem_hencky-typ}
	For the important class of \emph{Hencky type} isotropic elastic energy functions \cite{agn_neff2014axiomatic,Hencky1928,Hencky1929,agn_neff2015geometry,agn_neff2015exponentiatedI,agn_neff2015exponentiatedII,nedjar2017finite,agn_montella2015exponentiated}
	\begin{equation}
	\label{eq:henckyTypeEnergies}
		W(F) = \mathcal W\bigl(\norm{\dev\log V}^2,\ \left|\tr\log V\right|^2\bigr)\,,
	\end{equation}
 	the requirement of tension-compression symmetry is always satisfied. Thus for every energy of the form \eqref{eq:henckyTypeEnergies}, every finite pure shear stretch $V_\alpha$ corresponds to a Cauchy pure shear stress $\sigmahat(B)$. These energies include the classical Hencky energy \cite{Hencky1929}
 	\begin{equation}\label{eq:henckyEnergy}
	 	\WH(F) = \mu\,\norm{\dev_3 \log U}^2 + \frac \kappa 2\left(\tr(\log U)\right)^2\ =\ \mu\,\norm{\log U}^2 + \frac \Lambda 2\left(\tr(\log U)\right)^2
 	\end{equation}
 	as well as the so-called \emph{exponentiated Hencky energy} \cite{agn_neff2015exponentiatedI} 
	\[
		\WeH(U)=\frac{\mu}{k}\.e^{k\norm{\dev\log V}^2}+\frac{\kappa}{2\.\widehat k}\.e^{\widehat k [\tr\log V]^2}\,.
	\]
\end{example}

A particularly interesting application of Theorem \ref{theorem:shearCompatibilityMainResult} is the class of energy functions which exhibit an \emph{additive isochoric-volumetric split}, i.e.\ energies of the form
\[
	W(F) = \Wiso\left(\frac{F}{(\det F)^{\afrac13}}\right) + f(\det F)\,.
\]
In terms of the principal invariants, such energy functions can also be expressed as
\[
	W(I_1,I_2,I_3)=\Wiso(I_1I_3^{-\frac{1}{3}},I_2I_3^{-\frac{2}{3}})+f(I_3)\,.
\]
Note that the condition $f'(1)=0$ is necessarily satisfied for such energy functions if the reference configuration $F=\id$ is stress free.

\begin{corollary}
\label{corollary:volIsoSplit}
	Let $W$ be an isotropic elastic energy with an additive isochoric-volumetric split and a stress-free reference configuration. If $\Wiso$ is tension-compression symmetric\footnote{%
		Note that tension-compression symmetry of $W(F)$ implies tension-compression symmetry of $\Wiso(J_1,J_2)$ because $W(F)=W(F\inv)\iff W(\frac{I_2}{I_3},\frac{I_1}{I_3},\frac{1}{I_3})\overset{I_3=1}{\implies}W(I_1,I_2,1)=W(I_2,I_1,1)\iff\Wiso(J_1,J_2)=\Wiso(J_2,J_1)$.%
		}, i.e.\ if $\Wiso(J_1,J_2)=\Wiso(J_2,J_1)$ for all $F\in\GLp(3)$, then $\sigmahat(B)$ is a pure shear stress for every finite pure shear stretch $V=V_\alpha$.
\end{corollary}

\begin{example}
	As an example for the application of Corollary \ref{corollary:volIsoSplit}, consider a special format of a slightly compressible Mooney-Rivlin energy 
	\begin{align}
		W(I_1,I_2,I_3)=\frac{\mu}{4}\.\left((I_1I_3^{-\frac{1}{3}}-3)+(I_2I_3^{-\frac{2}{3}}-3)\right)+f(I_3)\,,
	\end{align}
	consisting of the tension-compression symmetric isochoric term $\Wiso(J_1,J_2)=\frac{\mu}{4}\.\left((J_1-3)+(J_2-3)\right)$ and arbitrary volumetric term $f(I_3)$ with $f'(1)=0$.\qed
\end{example}

Another important class of functions to (some of) which Theorem \ref{theorem:shearCompatibilityMainResult} is applicable are the so-called \emph{Valanis-Landel energies} of the form
\[
	W(\lambda_1,\lambda_2,\lambda_3) = \sum_{i=1}^3 w(\lambda_i) + f(\det F)
\]
with functions $w,f\col\Rp\to\R$. These energies were originally introduced by Valanis and Landel \cite{valanis1967} as a model for incompressible materials, but are often coupled with a volumetric term and applied to the compressible case as well.

\begin{corollary}\label{corollary:pureShearValanisLandel}
	Let $W$ be an isotropic elastic energy of the \emph{generalized Valanis-Landel form}, i.e.
	\begin{equation}
		W(F) \;=\; \sum_{i=1}^3 w(\lambda_i) \;+\; f(\det F) \;=\; w(\lambda_1) + w(\lambda_2) + w(\lambda_3) + f(\lambda_1\,\lambda_2\,\lambda_3)\label{eq:definitionValanisLandelGeneralForm}
	\end{equation}
	for all $F\in\GLp(3)$ with singular values $\lambda_1,\lambda_2,\lambda_3$, where $w,f\col\Rp\to\R$ are differentiable functions. If the reference configuration is stress free and $w$ is tension-compression symmetric, i.e.\ $w(\lambda)=w(\tfrac{1}{\lambda})$ for all $\lambda\in\Rp$, then $\sigmahat(B)$ is a pure shear stress for every finite pure shear stretch $V=V_\alpha$.
\end{corollary}
\begin{proof}
	It is easy to verify that a function $W\col\GLp(3)$ with $W(F)=\sum_{i=1}^3 w(\lambda_i)$ for all $F$ with singular values $\lambda_i$ is tension-compression symmetric if and only if $w(\frac1\lambda)=w(\lambda)$ and that the requirement of a stress-free reference configuration implies $f'(1)=0$, thus Theorem \ref{theorem:shearCompatibilityMainResult} is directly applicable.
\end{proof}

The condition $w(\frac1\lambda)=w(\lambda)$ is satisfied for a large variety of Valanis-Landel type elastic energy potentials, including, for example, the \emph{Ba\v{z}ant energy} \cite{bazant1998}
\[
	W(F) =\frac{\mu}{4}\. \norm{B-B\inv}^2 =\frac{\mu}{4}\. \sum_{i=1}^3 \left(\lambda_i^2-\frac{1}{\lambda_i^2}\right)^2
\]
or, again, the logarithmic-quadratic Hencky energy \eqref{eq:henckyEnergy}, which can be expressed in the form \eqref{eq:definitionValanisLandelGeneralForm} with $w(\lambda)=\mu\,\log^2(\lambda)$ and $f(d)=\frac \Lambda 2(\log d)^2$.

Of course, without any additional conditions, a pure shear stretch does \emph{not} necessarily induce a Cauchy pure shear stress for arbitrary constitutive laws.
\begin{example}
	Consider the \emph{Blatz-Ko} type energy \cite{horgan1996remarks}
	\begin{align}
		W(F)=\frac{\mu}{2}\left(\norm{F}^2+\frac{2}{\det F}-5\right)=\frac{\mu}{2}\left(I_1+\frac{2}{\sqrt{I_3}}-5\right)
	\end{align}
	with the corresponding Cauchy stress response
	\begin{align}
		\sigmahat(B)=\beta_0\.\id+\beta_1\.B=\frac{\mu}{\sqrt{I_3}}\.B-\frac{\mu}{I_3}\.\id=\frac{\mu}{I_3}\left(\sqrt{I_3}\.B-\id\right)\,.
	\end{align}
	For a deformation of the form \eqref{eq:shearStretchStress}, we find
	\[
		\sigmahat(B)%=\mu\left(1\.\diag(\lambda\,,\frac{1}{\lambda}\,,1)-\id\right)
		=\mu\.\diag(\lambda-1\,,\frac{1}{\lambda}-1\,,0)\,.
	\]
	Thus, a finite pure shear stretch $V_\alpha$ can only correspond to a Cauchy pure shear stress if
	\[
		\lambda-1=1-\frac{1}{\lambda}\quad\iff\quad\lambda+\frac{1}{\lambda}=2\quad\iff\quad\lambda= 1\,,
	\]
	i.e.\ only in the trivial case $V=\id$.
	
	Furthermore, if $\sigma$ is a (diagonal) Cauchy pure shear stress of the form $\sigma=\diag(s,-s,0)$, then the eigenvalues of the corresponding stretch $V=\diag(\lambda_1,\lambda_2,\lambda_3)$ satisfy the equalities
	\begin{alignat*}{4}
		&&\diag(s,-s,0)&\mathrlap{%
			=\frac{\mu}{I_3}\left(\sqrt{I_3}\.B-\id\right)=\frac{\mu}{I_3}\left(\sqrt{I_3}\.\diag(\lambda_1^2,\lambda_2^2,\lambda_3^2)-\id\right)%
		}\\
		\implies&&\qquad 0&=\sqrt{I_3}\.\lambda_3^2-1\qquad&&\text{and}&\qquad \sqrt{I_3}\.\lambda_1^2-1&=-\left(\sqrt{I_3}\lambda_2^2-1\right)\\
		\iff&&\qquad 1&=\sqrt{I_3}\.\lambda_3^2\qquad&&\text{and}&\qquad \sqrt{I_3}\left(\lambda_1^2+\lambda_2^2\right)&=2
	\end{alignat*}
	If the stretch is volume preserving ($I_3=1$), then the first equation yields $\lambda_3=1$ (planarity). Similarly, if the stretch is planar, then the first equation reads $I_3=1$, i.e.\ the stretch is volume preserving Since both properties can only be satisfied simultaneously for $V=\id$, any non-trivial deformation corresponding to a Cauchy pure shear stress is \emph{neither} volume-preserving \emph{nor} planar.\qed
\end{example}
%
%
%%%%%%%%%%
\subsection{Pure shear stretch induced by pure shear stress}
\label{section:shearCompatibilityReverse}
We now return to the question whether (or, more precisely, under which conditions) pure shear Cauchy stress induces a pure shear stretch $V_\alpha$. Note again that while Theorem \ref{theorem:shearCompatibilityMainResult} and the subsequent corollaries ensure that every pure shear stretch $V_\alpha$ induces a Cauchy pure shear stress, additional assumptions on the energy function are required to ensure the reverse implication since the Cauchy stress response is generally not invertible. The following theorem uses the condition of \emph{Hill's (strict) inequality} \cite{hill1970}
\begin{equation}\label{eq:hillStrict}
	\iprod{\tau(V_1)-\tau(V_2),\, \log(V_1)-\log(V_2)} > 0 \qquad\text{for all }\; V_1,V_2\in\PSymn\,,\; V_1\neq V_2\,,
\end{equation}
where $\tau(V)=\det(V)\cdot\sigma(V)=\det(V)\cdot\sigmahat(V^2)$ denotes the \emph{Kirchhoff stress} corresponding to the stretch $V$. For hyperelastic materials, this inequality is equivalent to the strict convexity of the mapping $X\mapsto W(\exp(X))$ on $\Symn$.

\begin{theorem}
\label{theorem:reverseCondition}
	Let $W$ be a sufficiently smooth isotropic elastic energy satisfying the conditions \eqref{eq:CompatibilityConditions2}. Furthermore, assume that $W$ is $p$-coercive for some $p\geq1$, i.e.\ $W(F)\geq d+c\cdot\norm{F}^p$ for some $d\in\R$ and $c>0$, and that $W$ satisfies Hill's (strict) inequality \eqref{eq:hillStrict}. Then $\sigma(V)$ is a pure shear stress tensor \emph{if and only if} $V$ is a pure shear stretch.
\end{theorem}
\begin{proof}
	By assumption, if $V$ is a pure shear stretch, then $\sigma(V)$ is a pure shear stress. In this case, the corresponding Kirchhoff stress $\tau(V)=\det(V)\cdot\sigma(V)$ is a pure shear stress as well. For $\lambda>0$, define $s^\tau(\lambda)$ by
	\[
		\tau(\diag(\lambda,\tfrac1\lambda,1)) = (s^\tau(\lambda),-s^\tau(\lambda),0)\,.
	\]
	We first show that $\lambda\mapsto s^\tau(\lambda)$ is a surjective mapping from $(0,\infty)$ to $\R$. Due to \eqref{eq:ogdenFormula},
	\[
		s^\tau(\lambda) = \lambda\cdot \pdd{W}{\lambda_1} \left(\lambda,\frac1\lambda,1\right)
		\qquad\text{and}\qquad
		-s^\tau(\lambda) = \frac1\lambda\cdot \pdd{W}{\lambda_2} \left(\lambda,\frac1\lambda,1\right)
	\]
	and thus
	\begin{align*}
		\frac{d}{d\lambda} \, W\left(\lambda,\frac1\lambda,1\right) = \pdd{W}{\lambda_1}\left(\lambda,\frac1\lambda,1\right) - \frac{1}{\lambda^2}\, \pdd{W}{\lambda_2}\left(\lambda,\frac1\lambda,1\right) = \frac{2}{\lambda}\,s^\tau(\lambda)\,.
	\end{align*}
	Now, assume that $s^\tau(\lambda)\leq C$ for all $\lambda\in\R^+$ for some $C\in\R$. Then for all $\lambda_0>1$,
	\begin{equation}\label{eq:energyBoundedAssumption}
		W\left(\lambda_0,\frac{1}{\lambda_0},1\right) = W(1,1,1) + \int_1^{\lambda_0} \frac{d}{d\lambda} W\left(\lambda,\frac1\lambda,1\right) \,\intd{\lambda} \leq W(1,1,1) + \int_1^{\lambda_0} \frac{2\.C}{\lambda} \,\intd{\lambda} = W(1,1,1) + 2\.C\.\log(\lambda_0)\,,
	\end{equation}
	in contradiction to the coercivity of $W$, which implies
	\[
		W\left(\lambda_0,\frac{1}{\lambda_0},1\right) \geq d+\ctilde\cdot\dynnorm{\diag\left(\lambda_0,\frac{1}{\lambda_0},1\right)}_2^p = d+\ctilde\cdot\lambda_0^p
	\]
	for some appropriate constant $\ctilde>0$, where $\norm{\,.\,}_2$ denotes the spectral norm. Similarly, we find that $s^\tau(\lambda)$ cannot be bounded below for $\lambda\to0$. In particular, the mapping $\lambda\mapsto s^\tau(\lambda)$ from $(0,\infty)$ to $\R$ is surjective.
	
	Now let $V\in\PSymn$ such that for some $s\in\R$
	\[
		\sigma(V) = \matr{0&s&0\\s&0&0\\0&0&0}\,,\qquad \text{i.e.}\qquad \tau(V) = \matr{0&s^\tau&0\\s^\tau&0&0\\0&0&0}\qquad \text{with }\; s^\tau = \frac{s}{\det V}\,.
	\]
	Choose $\lambda>0$ such that $s^\tau=s^\tau(\lambda)$. Then
	\begin{equation}\label{eq:tauIsotropy}
		\tau(V) = Q\.\diag(s^\tau,-s^\tau,0)\.Q^T = Q\,\tau(\diag(\lambda,\tfrac1\lambda,1))\.Q^T = \tau(Q\.\diag(\lambda,\tfrac1\lambda,1)\.Q)
	\end{equation}
	with $Q\in\SO(3)$ given by \eqref{eq_Q}. Since Hill's strict inequality implies that the mapping $\log V \mapsto \tau(V)$ is strictly monotone and hence injective, the mapping $V \mapsto \tau(V)$ is injective as well. Therefore, \eqref{eq:tauIsotropy} implies that $V=Q\.\diag(\lambda,\tfrac1\lambda,1)\.Q$ is a pure shear stretch.
\end{proof}

\begin{remark}
	Note that in the proof of Proposition \ref{theorem:reverseCondition}, the coercivity of the energy is used only to contradict inequality \eqref{eq:energyBoundedAssumption}. It is easy to see that the condition of $p$-coercivity can be replaced by the weaker requirement that $W$ has \enquote{stronger-than-logarithmic} growth, i.e.\ $\log(\norm{F})\in o(W(\norm{F}))$. In particular, the result is also applicable to the classical quadratic-logarithmic \emph{Hencky energy} \eqref{eq:henckyEnergy} which is not $p$-coercive for any $p\geq1$, but does satisfy $W(F)\geq c\.\log^2(\norm{F})$ for some appropriate $c>0$.\footnote{The correspondence between pure shear stretch and pure shear Cauchy stress for the classical Hencky model was shown earlier by Vall{\'e}e \cite{vallee1978}.}
\end{remark}
\begin{remark}
	In particular, the equalities \eqref{eq:CompatibilityConditions2} required in Theorem \ref{theorem:reverseCondition} hold if the conditions of Theorem \ref{theorem:shearCompatibilityMainResult} are satisfied.
\end{remark}
%
%
%
%
%%%%%%%%%%	
\section*{Conclusion}
While the incompatibility between simple shear and Cauchy pure shear stress described by Destrade et al.\ \cite{destrade2012} and Moon and Truesdell \cite{moon1974interpretation} as well as Mihai and Goriely \cite{mihai2011positive} is due to the difference between the \emph{principal axes} of strain and stress, i.e.\ the eigenspaces of $B=F_\gamma F_\gamma^T$ and $\sigma=\sigma^s$, the question whether or not a pure shear Cauchy stress corresponds to a pure shear stretch $V_\alpha$ depends only on the \emph{principal stretches} and \emph{principal stresses}, i.e.\ the eigenvalues of $B$ and $\sigma$. In particular, this latter question of compatibility is a matter of the constitutive law and depends on the specific choice of a stress response function. Physical reasons suggest that a correspondence between pure shear stretch and pure shear stress is plausible, as further demonstrated by the criteria and examples considered here. As a substitute for the concept of simple shear in finite elasticity, we introduced the notion of an idealized \enquote{finite simple shear} deformation (gradient) which not only corresponds to a pure shear stretch tensor, but is also coaxial to pure shear stress tensors.

In order to avoid the confusion arising from the apparent similarities to the linear case, we suggest to use the term \enquote{\emph{simple glide}} instead of \enquote{simple shear} for deformations $F_\gamma$ in finite elasticity which emphasizes the idea of a body consisting of horizontally displaceable layers.
%
%
%
%%%%%%%%%%
\section*{Acknowledgement}
We thank David J. Steigmann (UC Berkeley) and Ray W. Ogden (U Glasgow) for discussions on the concept of pure shear stress.

%
%
%
%
%%%%%%%%%%
\footnotesize
\section{References}
\printbibliography[heading=none]
%
%
%
%
%%%%%%%%%%
\begin{appendix}
%
%
%
%
%%%%%%%%%%
\section{Biot pure shear stress}\label{sec:Biot}
The following lemmas and theorems discuss the connection between Biot pure shear stress $\TBiot$ and right finite simple shear deformations $F_\alpha$ analogues to Cauchy pure shear stress. For each of the results presented here, we also provide the reference to the respective counterpart for the Cauchy stress tensor.
	\begin{proposition}[Proposition \ref{prop_F_eindeutig_bestimmt_B}]\label{lem_F_eindeutig_bestimmt_C}
		Let $C\in\PSym(3)$ be given by \eqref{form_von_P} with $p>\abs{q}$ and $r>0$. Then $F \in \GLp(3)$ with $F^TF=C$ is uniquely determined by
		\begin{equation}\label{eq_f_triaxial_stretch_simple_shear_C}
			F =Q\,\diag(a,b,c)\.F_\gamma= Q\,\begin{pmatrix} a & 0 & 0 \\ 0 & b & 0 \\ 0 & 0 & c \end{pmatrix}\begin{pmatrix} 1 & \gamma & 0 \\ 0 & 1 & 0 \\ 0 & 0 & 1 \end{pmatrix}
		\end{equation}
		up to an arbitrary $Q\in\SO(3)$, where
		\begin{equation}
			a = \sqrt{p}\,,\qquad b = \sqrt{\frac{p^2-q^2}p}\,,\qquad c = \sqrt{r}\qquad\text{and}\qquad \gamma = \frac qp\label{eq:abcInpq2}
		\end{equation}
		or, in terms of the singular values $\lambda_1,\lambda_2,\lambda_3\in\R_+$ of $F$,
		\begin{equation}
			a=\sqrt{\frac{\lambda_1^2 + \lambda_2^2}2}\,,\qquad b=\lambda_1\lambda_2\sqrt{\frac 2{\lambda_1^2+\lambda_2^2}}\,,\qquad c=\lambda_3\qquad\text{and}\qquad \gamma=\frac{\lambda_1^2-\lambda_2^2}{\lambda_1^2+\lambda_2^2}\,.\label{abcInLambda2}
		\end{equation}
		In particular, $F$ is necessarily of the form \eqref{eq_f_triaxial_stretch_simple_shear_C} if the Biot stress tensor corresponding to the deformation gradient $F$ induced by an isotropic law of Cauchy elasticity is a pure shear stress of the form $\TBiot=s\.(e_1\otimes e_2+e_2\otimes e_1)$ with $s\in\R$.
	\end{proposition}
	\begin{proof}
		Let $\widetilde{F}=F_\gamma\.\diag(a,b,c)$ with $a,b,c$ given by \eqref{eq:abcInpq}. Then $a^2 + b^2\gamma^2 = b^2 = p$, $b^2\gamma = q$ and $c^2=r$, thus
		\begin{align*}
			\widetilde F^T\widetilde F &=  \begin{pmatrix} 1 & 0 & 0 \\ \gamma & 1 & 0 \\ 0 & 0 & 1 \end{pmatrix} \begin{pmatrix} a & 0 & 0 \\ 0 & b & 0 \\ 0 & 0 & c \end{pmatrix}\begin{pmatrix} a & 0 & 0 \\ 0 & b & 0 \\ 0 & 0 & c \end{pmatrix} \begin{pmatrix} 1 & \gamma & 0 \\ 0 & 1 & 0 \\ 0 & 0 & 1 \end{pmatrix}= \begin{pmatrix} a^2 & a^2\gamma & 0 \\ a^2\gamma & a^2\gamma^2 + b^2 & 0 \\ 0 & 0 & c^2 \end{pmatrix} =  \begin{pmatrix} p & q & 0 \\ q & p & 0 \\ 0 & 0 & r \end{pmatrix}\,.
		\end{align*}
	
		Due to \reflemma{lem_P_und_T_kommutieren}, $p=\frac 12(\lambda_1^2+\lambda_2^2)$, $q=\frac 12(\lambda_1^2-\lambda_2^2)$ and $r=\lambda_3^2$, which immediately implies $a=\sqrt{\frac{\lambda_1^2 + \lambda_2^2}2}$, $c=\lambda_3$ and $\gamma=\frac{\lambda_1^2-\lambda_2^2}{\lambda_1^2+\lambda_2^2}$. Moreover, $p+q=\lambda_1^2$ and $p-q=\lambda_2^2$, thus $p^2-q^2=\lambda_1^2\,\lambda_2^2$ and therefore $b=\lambda_1\lambda_2\sqrt{\frac 2{\lambda_1^2+\lambda_2^2}}$.
		
		Now let $B$ be given by \eqref{form_von_P} and consider an arbitrary $F\in\GL^+(3)$ with $F^TF=C=\widetilde F^T\widetilde F$. Since $F$ is uniquely determined by $F^TF$ up to a left-hand rotation, there exists $Q\in\SO(3)$ with $F=Q\.\widetilde F$, thus $F$ is of the form \eqref{eq_f_triaxial_stretch_simple_shear_C}.
	\end{proof}
	\begin{lemma}[Lemma \ref{lem_F_simple_finite_shear_B}]\label{lem_F_simple_finite_shear_C}
		Let the stretch $C=P$ be of the form \eqref{form_von_P_shear}. Then the deformation gradient $F \in \GLp(3)$ with $F^TF=C$ is of the form
		\begin{equation}
			F = Q\cdot\begin{pmatrix} \sqrt{\cosh(2\alpha)} & \frac{\sinh(2\alpha)}{\sqrt{\cosh(2\alpha)}} & 0 \\ 0 & \frac 1{\sqrt{\cosh(2\alpha)}} & 0 \\ 0 & 0 & 1\end{pmatrix}= Q\cdot\frac 1{\sqrt{\cosh(2\alpha)}}\begin{pmatrix} \cosh(2\alpha) & \sinh(2\alpha) & 0 \\ 0 & 1 & 0 \\ 0 & 0 & \sqrt{\cosh(2\alpha)}\end{pmatrix}\,.
		\end{equation}
		with arbitrary $Q\in\SO(3)$.
	\end{lemma}
	\begin{proof}
		We use Proposition \ref{lem_F_eindeutig_bestimmt_C} and write $\lambda_1=e^\alpha$, $\lambda_2=e^{-\alpha}$, $\lambda_3=1$. Then
		\begin{align*}
			\gamma&=\frac{\lambda_1^2-\lambda_2^2}{\lambda_1^2+\lambda_2^2}  = \frac{e^{2\alpha}-e^{-2\alpha}}{e^{2\alpha}+e^{-2\alpha}} = \tanh(2\alpha)\,,\qquad a = \sqrt{\frac{\lambda_1^2+\lambda_2^2}2}  = \sqrt{\frac{e^{2\alpha} + e^{-2\alpha}}2} = \sqrt{\cosh(2\alpha)}\,,\\
			b &=  \lambda_1\lambda_2\sqrt{\frac 2{\lambda_1^2+\lambda_2^2}}  = e^{\alpha}e^{-\alpha}\cdot\sqrt{\frac 2{e^{2\alpha} + e^{-2\alpha}}} = \frac 1{\sqrt{\cosh(2\alpha)}}\qquad\text{and}\qquad	c = \lambda_3  = 1
		\end{align*}
		and thus
		\begin{align*}
			F &=  Q\begin{pmatrix} a & 0 & 0 \\ 0 & b & 0 \\ 0 & 0 & c \end{pmatrix}\begin{pmatrix} 1 & \gamma & 0 \\ 0 & 1 & 0 \\ 0 & 0 & 1 \end{pmatrix}= Q\begin{pmatrix} \sqrt{\cosh(2\alpha)} & 0 & 0 \\ 0 & \frac 1{\sqrt{\cosh(2\alpha)}} & 0 \\ 0 & 0 & 1 \end{pmatrix}\begin{pmatrix} 1 & \tanh(2\alpha) & 0 \\ 0 & 1 & 0 \\ 0 & 0 & 1 \end{pmatrix}\\
			&= Q\cdot\,\frac 1{\sqrt{\cosh(2\alpha)}}\begin{pmatrix} \cosh(2\alpha) & \sinh(2\alpha) & 0 \\ 0 & 1 & 0 \\ 0 & 0 & \sqrt{\cosh(2\alpha)} \end{pmatrix}\,.\qedhere
		\end{align*}
	\end{proof}
	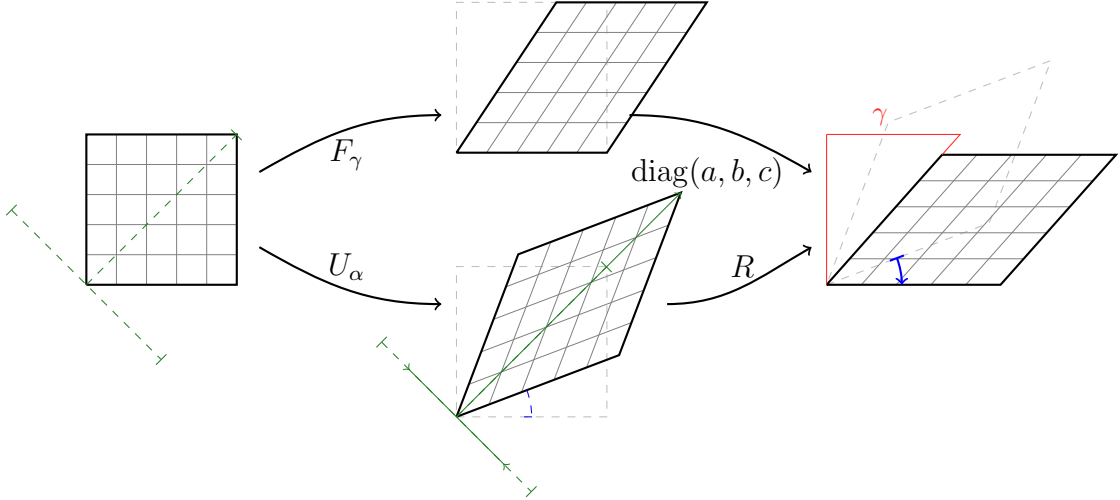
\begin{figure}[h!]
	    \centering
	%    \tikzRemake
	    \tikzsetnextfilename{PolarDecompositionFiniteSimpleShear2}
	    \begin{tikzpicture}
	        \input{tikz/tikz_PolarDecompositionFiniteSimpleShear2.tex}
	    \end{tikzpicture}
	    \caption{Right finite simple shear deformation with amount of shear $\gamma=\sinh(2\alpha)$.}
	\end{figure} 
	\begin{theorem}[Theorem \ref{theorem:main}]\label{theorem:main2}
		Any shear deformation\textsuperscript{\ref{footnote}} $F\in\GLp(3)$ that results in a (non-trivial) pure shear Biot stress tensor $\TBiot$ is a \emph{right finite simple shear deformation}
		\begin{equation}
			F_\alpha = \frac 1{\sqrt{\cosh(2\alpha)}}\matr{ \cosh(2\alpha) & \sinh(2\alpha) & 0 \\ 0 & 1 & 0 \\ 0 & 0 & \sqrt{\cosh(2\alpha)} }\qquad\text{with $\alpha\in\R$}\,.
		\end{equation}
	\end{theorem}
	\begin{proof}[Proof of Theorem \ref{theorem:main2}]
		Let $\TBiot$ be a pure shear stress tensor with $s\neq 0$ and $C\in\PSym(3)$ with $\lambda_1=\lambda$, $\lambda_2=\frac 1\lambda$, $\lambda_3=1$ for arbitrary $\lambda\in\R_+$ with $\lambda_1^2,$ $\lambda_2^2,$ $\lambda_3^2$ as the eigenvalues of $C$. By \reflemma{lem_form_von_p_kommutieren_mit_t}, the tensors $C$ and $\hTBiot(C)$ commute if and only if $C=U^2$ is of the form \eqref{form_von_P_shear}, i.e.\ if and only if
		\begin{align}
			U = \matr{\cosh(\alpha)&\sinh(\alpha)&0\\\sinh(\alpha)&\cosh(\alpha)&0\\0&0&1}=\exp\matr{0&\alpha&0\\\alpha&0&0\\0&0&0}\qquad\text{with $\alpha\in\R$}
		\end{align}
		is a finite pure shear stretch.
	
		For an isotropic elasticity law the right Cauchy-Green deformation tensors $C=F^TF$ commutes with $\hTBiot(C)$.
		Thus if $\hTBiot(C)$ is a Biot pure shear stress, then $C=F^TF$ is of the form \eqref{form_von_P_shear} (i.e.\ $U$ is a finite pure shear stretch). By \reflemma{lem_F_simple_finite_shear_C}, the deformation gradient $F$ is determined up to an arbitrary right-hand side rotation $Q\in\SO(3)$, and it remains to observe that exactly for $Q=\id$, the resulting deformation $F_\alpha$ satisfies the condition of ground parallelism (i.e.\ has eigenvectors $e_1$ and $e_3$).
	\end{proof}
	\begin{remark}
		Since in isotropic elasticity, the \emph{second Piola-Kirchhoff stress tensor} $S^2$ commutes with the right Cauchy-Green deformation tensor $C=F^TF$ as well, the above results still hold if Biot stress $\TBiot$ is replaced by $S^2$.
	\end{remark}
	It might also be of interest to investigate conditions on the energy $W$ which ensure that the Biot stress $\TBiot$ corresponding to a finite pure shear stretch $U_\alpha$ is a pure shear stress. Here we state, without proof, a result for hyperelastic models with energy functions of Valanis-Landel type.
	\begin{proposition}[Corollary \ref{corollary:pureShearValanisLandel}]\label{prop:BiotPureShearValanisLandel}
		Let $W$ be an isotropic elastic energy of the \emph{generalized Valanis-Landel form}, i.e.
		\[
			W(F) \;=\; \sum_{i=1}^3 w(\lambda_i) \;+\; f(\det F) \;=\; w(\lambda_1) + w(\lambda_2) + w(\lambda_3) + f(\lambda_1\,\lambda_2\,\lambda_3)
		\]
		for all $F\in\GLp(3)$ with singular values $\lambda_1,\lambda_2,\lambda_3$, where $w,f\col\Rp\to\R$ are differentiable functions. If the reference configuration is stress free and $w$ satisfies $w'\left(\frac1 \lambda\right) = -w'(\lambda)$ for all $\lambda\in\Rp$, then $\hTBiot(C)=\hTBiot(U^2)$ is a pure shear stress for every finite pure shear stretch $U=U_\alpha$.
	\end{proposition}
	For example, the requirement $w'\left(\frac1 \lambda\right) = -w'(\lambda)$ from Proposition \ref{prop:BiotPureShearValanisLandel} is satisfied by the function $w$ originally proposed by Valanis and Landel as a model for incompressible materials \cite[eq.\ (48)]{valanis1967}, which they defined via the equality $w'(\lambda)=2\mu\.\log(\lambda)$, i.e.\ (assuming $W(\id)=0$)
	\[
		w(t) = 2\mu\,(t\,(\log (t) - 1) + 1)\,,
	\]
	which corresponds to the energy
	\begin{equation}
		W(U) = 2\mu\,\iprod{U,\; \log U-\id} + 3\,.
	\end{equation}
	If this energy function is applied in the compressible case without an additional volumetric term, it induces \emph{Becker's law of elasticity} \cite{agn_neff2014rediscovering} in the lateral contraction free case (zero Poission's ratio). This constitutive law was first described by the geologist G.F.\ Becker \cite{becker1893}, who deduced it from a number of axioms including the equivalence of finite pure shear stretch $U$ and Biot pure shear stress; note that for Becker's stress-strain relation
	\begin{align}
		\TBiot(U)=2\mu\.\log U+\lambda\.\tr(\log U)\.\id\,,
	\end{align}
	every finite pure shear stretch $U=U_\alpha$ corresponds to a Biot pure shear stress and vice versa (cf.\ \cite{vallee1978}), even in the non-hyperelastic case $\lambda\neq0$. In fact, this compatibility property was one of Becker's major motivations for introducing his logarithmic constitutive model \cite{becker1893,agn_neff2014rediscovering}.
\section{Linearization of finite simple shear}\label{appendix:LinearShear}
In order to verify that the notions of finite simple shear and finite pure shear stretch (as well as the rotation $R$ given in \eqref{eq:polarFiniteShearFactors}) are compatible to the corresponding concepts of shear in linear elasticity via the identification $\gamma=2\.\alpha$, simply compare the linearizations
	\begin{align*}
		&\frac{1}{\sqrt{\cosh(2\.\alpha)}}\, \matr{1&\sinh(2\.\alpha)&0\\ 0&\cosh(2\.\alpha)&0\\ 0&0&\sqrt{\cosh(2\.\alpha)}} = \matr{1&2\.\alpha&0\\0&1&0\\0&0&1} + \mathcal O(\alpha^2)\,,\\
		&\frac{1}{\sqrt{\cosh(2\.\alpha)}}\, \matr{\cosh(2\.\alpha)&\sinh(2\.\alpha)&0\\ 0&1&0\\ 0&0&\sqrt{\cosh(2\.\alpha)}} = \matr{1&2\.\alpha&0\\0&1&0\\0&0&1} + \mathcal O(\alpha^2)\,,\\
		&V_\alpha = \begin{pmatrix} \cosh(\alpha) & \sinh(\alpha) & 0 \\ \sinh(\alpha) & \cosh(\alpha) & 0 \\ 0 & 0 & 1 \end{pmatrix} = \id + \matr{ 0 & \alpha & 0 \\ \alpha & 0 & 0 \\ 0 & 0 & 0} + \mathcal O(\alpha^2)
	\intertext{and}
		&R = \frac 1{\sqrt{\cosh(2\alpha)}}\begin{pmatrix} \cosh(\alpha) & \sinh(\alpha) & 0 \\ -\sinh(\alpha) & \cosh(\alpha) & 0 \\ 0 & 0 & \sqrt{\cosh(2\alpha)} \end{pmatrix} = \id + \matr{ 0 & \alpha & 0 \\ -\alpha & 0 & 0 \\ 0 & 0 & 0} + \mathcal O(\alpha^2)
	\end{align*}
	to Definitions \ref{definition:linearSimpleShear} and \ref{def:linearShear}.
%
%
%
%
%
%
%%%%%%%%%%%%
\section{Shear monotonicity}
\label{appendix:shearMonotonicity}

Although simple shear is generally not a suitable concept for nonlinear elasticity if (Cauchy) pure shear stresses are concerned, it does occur in finite deformations under certain displacement boundary conditions in the context of so-called \emph{anti-plane shear deformations} of the form $\varphi(x_1,x_2,x_3)=(x_1\,,x_2\,,x_3+u(x_1,x_2))$. In the following, we discuss the concept of \emph{shear monotonicity}, which has been the subject of a recent contribution by Voss et al.\ \cite{voss2018aps}.

Recall from Remark \ref{remark:simpleShearBC} that for a simple shear $F=F_\gamma$, the left Cauchy-Green tensor $B=FF^T$ is given by
\begin{align}
	B=F_\gamma F_\gamma^T=\left(\begin{matrix}1+\gamma^2&\gamma&0\\\gamma&1&0\\0&0&1\end{matrix}\right)
	\qquad\text{ and thus }\qquad
	B^{-1}=\left(\begin{matrix}1&-\gamma&0\\-\gamma&1+\gamma^2&0\\0&0&1\end{matrix}\right)\,,\label{eq:simpleshear}
\end{align}
and from \eqref{eq:CauchyInBetas} that in the isotropic case, the Cauchy stress tensor can always be expressed in the form
\begin{align}
	\sigma=\beta_0\.\id+\beta_1\.B+\beta_{-1}\.B^{-1}
\end{align}	
with scalar-valued functions $\beta_0$, $\beta_1$, $\beta_{-1}$ depending on the matrix invariants $I_1$, $I_2$, $I_3$ of the stretch tensor $B=FF^T$. In the case \eqref{eq:simpleshear} of simple shear, we find
\begin{align}
	I_1&=\tr B=3+\gamma^2\,,\qquad I_2=\tr(\Cof B)=\tr\left(\begin{matrix}1&-\gamma&0\\-\gamma&1+\gamma^2&0\\0&0&1\end{matrix}\right)=3+\gamma^2\,,\qquad I_3=\det B=1\,.
\end{align}
This allows us to consider $\beta_i$ as scalar-valued functions depending only on the amount of shear $\gamma\in\Rp.$ The Cauchy stress tensor $\sigma$ corresponding to a simple shear deformation $F_\gamma$ with the amount of shear $\gamma\in\R$ is given by
\begin{align}
	\sigma&=(\beta_0+\beta_1+\beta_{-1})\.\id+\left(\begin{matrix}\beta_1\.\gamma^2&(\beta_1-\beta_{-1})\.\gamma&0\\(\beta_1-\beta_{-1})\.\gamma&\beta_{-1}\.\gamma^2&0\\0&0&0\end{matrix}\right).\label{eq:CauchyStressSimpleShear}
\end{align}
Again, we can observe that a simple shear deformation does not map to a non-trivial Cauchy pure shear stress, since the Cauchy stress tensor of a simple shear deformation is a pure shear stress if and only if
\begin{align}
	\begin{matrix}
		\beta_0+(1+\gamma^2)\.\beta_1+\beta_{-1}&=0\,,\\
		\beta_0+\beta_1+(1+\gamma^2)\.\beta_{-1}&=0\,,\\
		\beta_0+\beta_1+\beta_{-1}&=0\,,
	\end{matrix}\qquad\implies\qquad\begin{matrix}
		\gamma^2\.(\beta_1-\beta_{-1})&=0\,,\\
		\beta_0+\beta_1+\beta_{-1}&=0\,,
	\end{matrix}
\end{align}
and for $\gamma\neq0$, the equality $\gamma^2\.(\beta_1-\beta_{-1})=0$ is satisfied if and only if $\beta_1-\beta_{-1}=0$, which results in an amount of shear stress $s=\gamma\.(\beta_1-\beta_{-1})=0$.

Now, consider the entry $\sigma_{12}(\gamma)=(\beta_1-\beta_{-1})\.\gamma$ as the amount of stress in $e_1$-direction acting on a surface perpendicular to $e_2$. From a physical point of view, increasing the amount of shear $\gamma$ should clearly lead to an increase in the amount of Cauchy shear stress; this material property is also known as shear monotonicity \cite{voss2018aps}. Since
\begin{align}
	\sigma_{12}(\gamma)&=(\beta_1-\beta_{-1})\.\gamma=2\gamma\left.\left(\frac{\partial W}{\partial I_1}+\frac{\partial W}{\partial I_2}\right)\right|_{I_1=I_2=3+\gamma^2,I_3=1}=\frac{d}{d\gamma}\.W(3+\gamma^2,3+\gamma^2,1)\,,
\end{align}
the condition for shear-monotonicity is given by
\begin{align}
	\frac{d}{d\gamma}\.\sigma_{12}(\gamma)=\frac{d^2}{(d\gamma)^2}\.W(3+\gamma^2,3+\gamma^2,1)>0\qquad\qquad\forall\,\gamma\geq  0\,,
\end{align}
which is equivalent to the convexity of the mapping $\gamma\mapsto W(3+\gamma^2,3+\gamma^2,1)$. This convexity property, known as \emph{APS-convexity}, is equivalent to the convexity of the restriction of the energy functional $W(F)$ to the set of APS-deformations (i.e.\ deformations $\varphi$ of the form $\varphi=(x_1,x_2,x_3+u(x_1,x_2))$).
In particular, APS-convexity is implied by rank-one convexity \cite{voss2018aps}, thus every Cauchy stress response induced by a rank-one convex energy is shear monotone.

\end{appendix}
\end{document}

%% file: tikz/tikz_simpleShearDeformation.tex
\begin{scope}[x  = {(1cm,0cm)},
                    y  = {(0cm,1cm)},
                    z  = {(0.5cm,0.5cm)},
                    scale = 1]
% style of faces
\definecolor{lightgray2}{rgb}{0.95,0.95,0.95}
\tikzset{facestyle/.style={fill=lightgray2,draw=black,thick,line join=round,opacity=0.65}}
% face "back" 

\begin{scope}[canvas is xy plane at z=0]
  \coordinate (a1) at (0,0);
  \coordinate (a2) at (2,0);
  \coordinate (a3) at (2,2);
  \coordinate (a4) at (0,2);

  \coordinate (a5) at (5,0);
  \coordinate (a6) at (7,0);
  \coordinate (a7) at (8,2);
  \coordinate (a8) at (6,2);

  \coordinate (ag) at (8,0);
  \coordinate (ah) at ({5},{2});
  
\end{scope}
\begin{scope}[canvas is xy plane at z=2]
  \coordinate (b1) at (0,0);
  \coordinate (b2) at (2,0);
  \coordinate (b3) at (2,2);
  \coordinate (b4) at (0,2);

  \coordinate (b5) at (5,0);
  \coordinate (b6) at (7,0);
  \coordinate (b7) at (8,2);
  \coordinate (b8) at (6,2);
\end{scope}

\begin{scope}[canvas is xy plane at z=1]
  \coordinate (c1) at (1.5,2.4);
  \coordinate (c2) at (4,3.5);
  \coordinate (c3) at (7,2.4);
\end{scope}

\path[facestyle] (b1) -- (b2) -- (b3) -- (b4) -- (b1);%hinten
\path[facestyle] (a1) -- (a2) -- (b2) -- (b1) -- (a1);%unten
\path[facestyle] (a1) -- (b1) -- (b4) -- (a4) -- (a1);%links
\path[facestyle] (a4) -- (a3) -- (b3) -- (b4) -- (a4) node[pos=0.4, left=0.2em] {1};%oben
\path[facestyle] (a2) -- (b2) -- (b3) -- (a3) -- (a2);%rechts
\path[facestyle] (a1) -- (a2)  node[pos=0.5, below] {1} -- (a3) -- (a4) -- (a1) node[pos=0.5, left] {1} ;%vorne

\path[facestyle] (b5) -- (b6) -- (b7) -- (b8) -- (b5);%hinten
\path[facestyle] (a5) -- (a6) -- (b6) -- (b5) -- (a5);%unten
\path[facestyle] (a5) -- (b5) -- (b8) -- (a8) -- (a5);%links
\path[facestyle] (a8) -- (a7) -- (b7) -- (b8) -- (a8) node[pos=0.4, left=0.2em] {1};%oben
\path[facestyle] (a6) -- (b6) -- (b7) -- (a7) -- (a6);%rechts
\path[facestyle] (a5) -- (a6) node[pos=0.5, below]{1} -- (a7) -- (a8) -- (a5);%vorne

\path[draw=black] (a6) -- (a7) -- (ag) node[pos=0.51, right=-0.3em] {$\left.\begin{matrix}\\[4.5em]\end{matrix}\right\}1$} -- (a6) node[pos=0.49,below=-0.2em] {$\underbrace{\hspace{2.7em}}_\gamma$};

\draw[->,thick] (c1) .. controls (c2) .. (c3) node[pos=0.5, above] {\large $F$};

\path[draw=black] (a5) -- (ah);
%\draw (a5) pic{carc=-30:30:2em} ;
%\draw (a5) ++(90:1.2) arc (90:{90-atan(tanh(\a))}:1.2) node[left=-0.1em]{$\overset{\phantom b}{\vartheta}$};
\draw (a5) ++(90:1.2) arc (90:64:1.2) node[,left=0.3em]{$\overset{\phantom b}{\vartheta}$};
\path[draw=black,dashed] (a3) -- (a7);

\end{scope}

%% file: tikz/tikz_shear.tex
\begin{scope}[x  = {(1cm,0cm)},
                    y  = {(0cm,1cm)},
                    z  = {(0.5cm,0.5cm)},
                    scale = 1]
% style of faces
\definecolor{lightgray2}{rgb}{0.95,0.95,0.95}
\definecolor{lightred}{rgb}{0.9,0.8,0.8}
\definecolor{darkred}{rgb}{0.5,0,0}
\definecolor{lightgreen}{rgb}{0.8,0.9,0.8}
\definecolor{lightgreen2}{rgb}{0.65,0.8,0.65}
\definecolor{darkgreen}{rgb}{0,0.5,0}
\tikzset{facestylered/.style={fill=lightred,draw=black,thick,line join=round,opacity=0.65}}
\tikzset{facestylegreen/.style={fill=lightgreen,draw=black,thick,line join=round,opacity=0.65}}
\tikzset{facestylegreen2/.style={fill=lightgreen2,draw=black,thick,line join=round,opacity=0.65}}
\tikzset{facestyle/.style={fill=lightgray2,draw=black,thick,line join=round,opacity=0.65}}
\tikzset{facestylemini/.style={draw=black,thick,line join=round,opacity=0.65}}
% face "back" 

\def\tiefe{5.2}
\def\breite{8}
\def\hoehe{0.9}
\def\pfeilbreite{0.3}

\begin{scope}[canvas is xy plane at z=0]
  \coordinate (a1) at (0,0);
  \coordinate (a2) at ({\breite},0);
  \coordinate (a3) at ({\breite},{\hoehe});
  \coordinate (a4) at (0,{\hoehe});
\end{scope}

\begin{scope}[canvas is xy plane at z={\tiefe}]
  \coordinate (a5) at (0,0);
  \coordinate (a6) at ({\breite},0);
  \coordinate (a7) at ({\breite},{\hoehe});
  \coordinate (a8) at (0,{\hoehe});
\end{scope}

\begin{scope}[canvas is xy plane at z={\tiefe/2-\hoehe/2}]
  \coordinate (t1) at ({\breite/2-\breite/20},{0});
  \coordinate (t2) at ({\breite/2-\breite/20+\hoehe},{0});
  \coordinate (t3) at ({\breite/2-\breite/20+\hoehe},{\hoehe});
  \coordinate (t4) at ({\breite/2-\breite/20},{\hoehe});
\end{scope}

\begin{scope}[canvas is xy plane at z={\tiefe/2+\hoehe/2}]
  \coordinate (t5) at ({\breite/2-\breite/20},{0});
  \coordinate (t6) at ({\breite/2-\breite/20+\hoehe},{0});
  \coordinate (t7) at ({\breite/2-\breite/20+\hoehe},{\hoehe});
  \coordinate (t8) at ({\breite/2-\breite/20},{\hoehe});
\end{scope}

\path[facestyle] (a5) -- (a6) -- (a7) -- (a8) -- (a5);%hinten
\path[facestylegreen] (a1) -- (a2) -- (a6) -- (a5) -- (a1);%unten
\path[facestyle] (a1) -- (a5) -- (a8) -- (a4) -- (a1);%links

\path[facestylemini,dotted] (t5) -- (t6) -- (t7) -- (t8) -- (t5);%hinten
\path[facestylegreen2,dotted] (t1) -- (t2) -- (t6) -- (t5) -- (t1);%unten
\path[facestylemini,dotted] (t1) -- (t5) -- (t8) -- (t4) -- (t1);%links
\path[facestylemini,dotted] (t2) -- (t6) -- (t7) -- (t3) -- (t2);%rechts
\path[facestylemini,dotted] (t1) -- (t2)  -- (t3) -- (t4) -- (t1);%vorne

\foreach \y in {1,2,3} {
	\begin{scope}[canvas is xy plane at z={\tiefe/4*\y}]
		\foreach \x in {1,2,3,4,5}  {
%			\draw[->,thick,darkred] ({\breite/5*\x-\breite/10},{\hoehe}) -- ({\breite/5*\x-\breite/10+\pfeilbreite},{\hoehe});
			\draw[->,thick,darkgreen] ({\breite/5*\x-\breite/10},0) -- ({\breite/5*\x-\breite/10-\pfeilbreite},0);
		}		
	\end{scope}
}

\begin{scope}[canvas is xy plane at z={\tiefe/2}]
%	\draw[->,thick,darkred]  ({\breite/2-0.55},{\hoehe+1.6}) -- ({\breite/2+0.65},{\hoehe+1.6}) node[pos=0.5,above] {~};
%	\draw[->,thick,darkgreen]  ({\breite/2+0.65},{-1.7}) -- ({\breite/2-0.55},{-1.7})  node[pos=0.5,below] {~};
\end{scope}

\begin{scope}[canvas is xy plane at z={\tiefe}]
	\draw[->,thick,darkred]  ({\breite/2-0.8},{\hoehe/2+\hoehe*1.1}) -- ({\breite/2+0.5},{\hoehe/2+\hoehe*1.1}) node[pos=0.5,above] {upper shear force};
\end{scope}

\begin{scope}[canvas is xy plane at z={0}]
	\draw[->,thick,darkgreen]  ({\breite/2+0.5},{\hoehe/2-\hoehe*1.1}) -- ({\breite/2-0.8},{\hoehe/2-\hoehe*1.1}) node[pos=0.5,below] {lower shear force};
\end{scope}

\path[facestylered] (a4) -- (a3) -- (a7) -- (a8) -- (a4);%oben
\path[facestyle] (a2) -- (a6) -- (a7) -- (a3) -- (a2);%rechts
\path[facestyle] (a1) -- (a2)  -- (a3) -- (a4) -- (a1);%vorne

\path[facestylered,dotted] (t4) -- (t3) -- (t7) -- (t8) -- (t4);%oben

\foreach \y in {1,2,3} {
	\begin{scope}[canvas is xy plane at z={\tiefe/4*\y}]
		\foreach \x in {1,2,3,4,5}  {
			\draw[->,thick,darkred] ({\breite/5*\x-\breite/10},{\hoehe}) -- ({\breite/5*\x-\breite/10+\pfeilbreite},{\hoehe});
%			\draw[->,thick,darkgreen] ({\breite/5*\x-\breite/10},0) -- ({\breite/5*\x-\breite/10-\pfeilbreite},0);
		}		
	\end{scope}

}

\def\eins{1.5}
\def\dist{3}

\begin{scope}[canvas is xy plane at z={0+\tiefe*0.2}]
  \coordinate (e0) at ({\breite+\dist},{0});
  \coordinate (e1) at ({\breite+\dist+\eins*1.4},{0});
  \coordinate (e2) at ({\breite+\dist},{\eins*1.4});
  \coordinate (f1) at ({\breite+\dist+\eins},{0});
  \coordinate (f2) at ({\breite+\dist},{\eins});
  \coordinate (f4) at ({\breite+\dist+\eins},{\eins});
\end{scope}

\begin{scope}[canvas is xy plane at z={1.5*\eins+\tiefe*0.2}]
  \coordinate (e3) at ({\breite+\dist},{0});
\end{scope}

\begin{scope}[canvas is xy plane at z={\eins+\tiefe*0.2}]
  \coordinate (f3) at ({\breite+\dist},{0});
  \coordinate (f5) at ({\breite+\dist+\eins},{0});
  \coordinate (f6) at ({\breite+\dist},{\eins});
  \coordinate (f7) at ({\breite+\dist+\eins},{\eins});
\end{scope}

\begin{scope}[canvas is xy plane at z={\eins*0.5+\tiefe*0.2}]
  \coordinate (p11) at ({\breite+\dist+\eins*0.7},{0});
  \coordinate (p12) at ({\breite+\dist+\eins*0.3},{0});
  \coordinate (p21) at ({\breite+\dist+\eins*0.3},{\eins});
  \coordinate (p22) at ({\breite+\dist+\eins*0.7},{\eins});
  \coordinate (p31) at ({\breite+\dist},{\eins*0.7});
  \coordinate (p32) at ({\breite+\dist},{\eins*0.3});
  \coordinate (p41) at ({\breite+\dist+\eins},{\eins*0.3});
  \coordinate (p42) at ({\breite+\dist+\eins},{\eins*0.7});
\end{scope}

\begin{scope}[canvas is xy plane at z={2.5*\eins+\tiefe*0.2}]
  \coordinate (g0) at ({\breite+\dist},{0});
  \coordinate (g1) at ({\breite+\dist+\eins},{0});
  \coordinate (g2) at ({\breite+\dist},{\eins});
  \coordinate (g3) at ({\breite+\dist+\eins},{\eins});
\end{scope}

\begin{scope}[canvas is xy plane at z={\tiefe*0.2-\eins*1.5}]
  \coordinate (h0) at ({\breite+\dist},{0});
  \coordinate (h1) at ({\breite+\dist+\eins},{0});
  \coordinate (h2) at ({\breite+\dist},{\eins});
  \coordinate (h3) at ({\breite+\dist+\eins},{\eins});
\end{scope}

\path[facestyle] (g0) -- (g1)  -- (g3) -- (g2) -- (g0);%hinten
\draw[thick,dashed] (g0)--(h0);
\draw[thick,dashed] (g1)--(h1);
\draw[thick,dashed] (g2)--(h2);
\draw[thick,dashed] (g3)--(h3);

\path[facestyle] (f3) -- (f5)  -- (f7) -- (f6) -- (f3);%hinten
\path[facestylegreen2] (f3) -- (f5)  -- (f1) -- (e0) -- (f3);%unten
\draw[->,very thick,darkgreen] (p11) -- (p12);
\path[facestyle] (f3) -- (e0)  -- (f2) -- (f6) -- (f3);%links
\draw[->,very thick] (p31) -- (p32);
\path[facestylered] (f2) -- (f4)  -- (f7) -- (f6) -- (f2);%oben
\draw[->,very thick,darkred] (p21) -- (p22);
\path[facestyle] (f1) -- (f4)  -- (f7) -- (f5) -- (f1);%rechts
\draw[->,very thick] (p41) -- (p42);
\path[facestyle] (e0) -- (f1)  -- (f4) -- (f2) -- (e0);%vorne

\path[facestyle] (h0) -- (h1)  -- (h3) -- (h2) -- (h0);%vorne

\draw[->,thick] (e0) -- (e1) node[pos=0.8,below] {$e_1$};
\draw[->,thick] (e0) -- (e2) node[pos=0.8,left] {$e_2$};
\draw[->,thick] (e0) -- (e3) node[pos=0.9, above right] {$e_3$};
%\draw[->,thick] (c1) .. controls (c2) .. (c3) node[pos=0.5, above] {\large $F$};

\end{scope}

%% file: tikz/tikz_leftFiniteSimpleShearDeformation.tex
\begin{scope}[x  = {(1cm,0cm)},
                    y  = {(0cm,1cm)},
                    z  = {(0.5cm,0.5cm)},
                    scale = 1]
% style of faces
\definecolor{lightgray2}{rgb}{0.95,0.95,0.95}
\tikzset{facestyle/.style={fill=lightgray2,draw=black,thick,line join=round,opacity=0.65}}
% face "back" 
\def\a{0.7}
\def\shift{4.5}

\begin{scope}[canvas is xy plane at z=0]
  \coordinate (a1) at (0,0);
  \coordinate (a2) at (2,0);
  \coordinate (a3) at (2,2);
  \coordinate (a4) at (0,2);

  \coordinate (a5) at ({\shift+0},{0});
  \coordinate (a6) at ({\shift+2/sqrt(cosh(\a))},{0});
  \coordinate (a7) at ({\shift+2*sinh(\a)/sqrt(cosh(\a))+2/sqrt(cosh(\a))},{2*sqrt(cosh(\a))});
  \coordinate (a8) at ({\shift+2*sinh(\a)/sqrt(cosh(\a))},{2*sqrt(cosh(\a))});

  \coordinate (ag) at ({\shift+2*sinh(\a)/sqrt(cosh(\a))+2/sqrt(cosh(\a))},{0});
  \coordinate (ah) at ({\shift},{2*sqrt(cosh(\a))});
  
  \coordinate (ab) at ({\shift+2*sinh(\a)/sqrt(cosh(\a))+2/sqrt(cosh(\a))},{2});

\end{scope}
\begin{scope}[canvas is xy plane at z=2]
  \coordinate (b1) at (0,0);
  \coordinate (b2) at (2,0);
  \coordinate (b3) at (2,2);
  \coordinate (b4) at (0,2);

  \coordinate (b5) at ({\shift+0},{0});
  \coordinate (b6) at ({\shift+2/sqrt(cosh(\a))},{0});
  \coordinate (b7) at ({\shift+2*sinh(\a)/sqrt(cosh(\a))+2/sqrt(cosh(\a))},{2*sqrt(cosh(\a))});
  \coordinate (b8) at ({\shift+2*sinh(\a)/sqrt(cosh(\a))},{2*sqrt(cosh(\a))});
\end{scope}

\begin{scope}[canvas is xy plane at z=1]
  \coordinate (c1) at (1.5,2.4);
  \coordinate (c2) at (4,3.5);
  \coordinate (c3) at ({\shift-0.1+2*sinh(\a)/sqrt(cosh(\a))+1/sqrt(cosh(\a))},{0.2+2*sqrt(cosh(\a))});
\end{scope}

\path[facestyle] (b1) -- (b2) -- (b3) -- (b4) -- (b1);%hinten
\path[facestyle] (a1) -- (a2) -- (b2) -- (b1) -- (a1);%unten
\path[facestyle] (a1) -- (b1) -- (b4) -- (a4) -- (a1);%links
\path[facestyle] (a4) -- (a3) -- (b3) -- (b4) -- (a4) node[pos=0.4, left=0.2em] {1};%oben
\path[facestyle] (a2) -- (b2) -- (b3) -- (a3) -- (a2);%rechts
\path[facestyle] (a1) -- (a2)  node[pos=0.5, below] {1} -- (a3) -- (a4) -- (a1) node[pos=0.5, left] {1} ;%vorne

\path[facestyle] (b5) -- (b6) -- (b7) -- (b8) -- (b5);%hinten
\path[facestyle] (a5) -- (a6) -- (b6) -- (b5) -- (a5);%unten
\path[facestyle] (a5) -- (b5) -- (b8) -- (a8) -- (a5);%links
\path[facestyle] (a8) -- (a7) -- (b7) -- (b8) -- (a8) node[pos=0.4, left=0.2em] {1};%oben
\path[facestyle] (a6) -- (b6) -- (b7) -- (a7) -- (a6);%rechts
\path[facestyle] (a5) -- (a6) node[pos=0.5, below]{\tiny$\frac{1}{\sqrt{\cosh (2\alpha)}}$} -- (a7) -- (a8) -- (a5);%vorne

\path[draw=black] (a6) -- (a7) -- (ag) node[pos=0.51, right=-0.3em] {\tiny$\left.\begin{matrix}\\[8.7em]\end{matrix}\right\}\sqrt{\cosh(2\alpha)}$} -- (a6) node[pos=0.49,below=-0.2em] {\tiny$\underbrace{\hspace{5.5em}}_\frac{\sinh(2\alpha)}{\sqrt{\cosh(2\alpha)}}$};

\path[draw=black] (a5) -- (ah);
\draw[->,thick] (c1) .. controls (c2) .. (c3) node[pos=0.5, above] {\large $F_\alpha=V_\alpha\.R$};
\path[draw=black,dashed] (a3) -- (ab);

%\draw (a5) pic{carc=-30:30:2em} ;
\draw (a5) ++(90:1.2) arc (90:{90-atan(tanh(\a))}:1.2) node[left=0.0em]{$\overset{\phantom b}{\vartheta^\star}$};
\end{scope}

%% file: tikz/tikz_rightFiniteSimpleShearDeformation.tex
\begin{scope}[x  = {(1cm,0cm)},
                    y  = {(0cm,1cm)},
                    z  = {(0.5cm,0.5cm)},
                    scale = 1]
% style of faces
\definecolor{lightgray2}{rgb}{0.95,0.95,0.95}
\tikzset{facestyle/.style={fill=lightgray2,draw=black,thick,line join=round,opacity=0.65}}
% face "back" 
\def\a{0.7}
\def\shift{4.5}

\begin{scope}[canvas is xy plane at z=0]
  \coordinate (a1) at (0,0);
  \coordinate (a2) at (2,0);
  \coordinate (a3) at (2,2);
  \coordinate (a4) at (0,2);

  \coordinate (a5) at ({\shift+0},{0});
  \coordinate (a6) at ({\shift+2*sqrt(cosh(\a))},{0});
  \coordinate (a7) at ({\shift+2*sinh(\a)/sqrt(cosh(\a))+2*sqrt(cosh(\a))},{2/sqrt(cosh(\a))});
  \coordinate (a8) at ({\shift+2*sinh(\a)/sqrt(cosh(\a))},{2/sqrt(cosh(\a))});

  \coordinate (ag) at ({\shift+2*sinh(\a)/sqrt(cosh(\a))+2*sqrt(cosh(\a))},{0});
  \coordinate (ah) at ({\shift},{2/sqrt(cosh(\a))});
  
  \coordinate (ab) at ({\shift+2*sinh(\a)/sqrt(cosh(\a))+2*sqrt(cosh(\a))},{2});

\end{scope}
\begin{scope}[canvas is xy plane at z=2]
  \coordinate (b1) at (0,0);
  \coordinate (b2) at (2,0);
  \coordinate (b3) at (2,2);
  \coordinate (b4) at (0,2);

  \coordinate (b5) at ({\shift+0},{0});
  \coordinate (b6) at  ({\shift+2*sqrt(cosh(\a))},{0});
  \coordinate (b7) at ({\shift+2*sinh(\a)/sqrt(cosh(\a))+2*sqrt(cosh(\a))},{2/sqrt(cosh(\a))});
  \coordinate (b8) at ({\shift+2*sinh(\a)/sqrt(cosh(\a))},{2/sqrt(cosh(\a))});
\end{scope}

\begin{scope}[canvas is xy plane at z=1]
  \coordinate (c1) at (1.5,2.4);
  \coordinate (c2) at (4,3.5);
  \coordinate (c3) at ({\shift+2*sinh(\a)/sqrt(cosh(\a))+1/sqrt(cosh(\a))},{0.2+2/sqrt(cosh(\a))});
\end{scope}

\path[facestyle] (b1) -- (b2) -- (b3) -- (b4) -- (b1);%hinten
\path[facestyle] (a1) -- (a2) -- (b2) -- (b1) -- (a1);%unten
\path[facestyle] (a1) -- (b1) -- (b4) -- (a4) -- (a1);%links
\path[facestyle] (a4) -- (a3) -- (b3) -- (b4) -- (a4) node[pos=0.4, left=0.2em] {1};%oben
\path[facestyle] (a2) -- (b2) -- (b3) -- (a3) -- (a2);%rechts
\path[facestyle] (a1) -- (a2)  node[pos=0.5, below] {1} -- (a3) -- (a4) -- (a1) node[pos=0.5, left] {1} ;%vorne

\path[facestyle] (b5) -- (b6) -- (b7) -- (b8) -- (b5);%hinten
\path[facestyle] (a5) -- (a6) -- (b6) -- (b5) -- (a5);%unten
\path[facestyle] (a5) -- (b5) -- (b8) -- (a8) -- (a5);%links
\path[facestyle] (a8) -- (a7) -- (b7) -- (b8) -- (a8) node[pos=0.4, left=0.2em] {1};%oben
\path[facestyle] (a6) -- (b6) -- (b7) -- (a7) -- (a6);%rechts
\path[facestyle] (a5) -- (a6) node[pos=0.5, below]{\tiny$\sqrt{\cosh (2\alpha)}$} -- (a7) -- (a8) -- (a5);%vorne

\path[draw=black] (a6) -- (a7) -- (ag) node[pos=0.51, right=-0.3em] {\tiny$\left.\begin{matrix}\\[7.3em]\end{matrix}\right\}\frac 1{\sqrt{\cosh(2\alpha)}}$} -- (a6) node[pos=0.49,below=-0.2em] {\tiny$\underbrace{\hspace{5.5em}}_\frac{\sinh(2\alpha)}{\sqrt{\cosh(2\alpha)}}$};

\path[draw=black] (a5) -- (ah);
\draw[->,thick] (c1) .. controls (c2) .. (c3) node[pos=0.5, above] {\large $F_\alpha=R\.V_\alpha$};
\path[draw=black,dashed] (a3) -- (ab);

%\draw (a5) pic{carc=-30:30:2em} ;
%\draw (a5) ++(90:1.2) arc (90:{90-atan(tanh(\a))}:1.2) node[left=0.2em]{$\overset{}{\vartheta}$};
\draw (a5) ++(90:1.2) arc (90:{90-atan(tanh(\a))}:1.4) node[,left=0.2em]{$\overset{\phantom b}{2\vartheta}$};

\end{scope}

%% file: tikz/tikz_pureShearStretch.tex
\begin{scope}[x  = {(1cm,0cm)},
                    y  = {(0cm,1cm)},
                    z  = {(0.5cm,0.5cm)},
                    scale = 1]
% style of faces
\definecolor{lightgray2}{rgb}{0.95,0.95,0.95}
\tikzset{facestyle/.style={fill=lightgray2,draw=black,thick,line join=round,opacity=0.65}}
% face "back" 
\def\a{0.3}

\begin{scope}[canvas is xy plane at z=0]
  \coordinate (a1) at (0,0);
  \coordinate (a2) at (2,0);
  \coordinate (a3) at (2,2);
  \coordinate (a4) at (0,2);

  \coordinate (a5) at ({5+0},{0});
  \coordinate (a6) at ({5+2*cosh(\a)},{2*sinh(\a)});
  \coordinate (a7) at ({5+2*exp(\a)},{2*exp(\a)});
  \coordinate (a8) at ({5+2*sinh(\a)},{2*cosh(\a)});

  \coordinate (ag) at ({5+2*cosh(\a)},{0});
  \coordinate (ah) at ({5},{2*cosh(\a)});

\end{scope}
\begin{scope}[canvas is xy plane at z=2]
  \coordinate (b1) at (0,0);
  \coordinate (b2) at (2,0);
  \coordinate (b3) at (2,2);
  \coordinate (b4) at (0,2);

  \coordinate (b5) at ({5+0},{0});
  \coordinate (b6) at ({5+2*cosh(\a)},{2*sinh(\a)});
  \coordinate (b7) at ({5+2*exp(\a)},{2*exp(\a)});
  \coordinate (b8) at ({5+2*sinh(\a)},{2*cosh(\a)});
\end{scope}

\begin{scope}[canvas is xy plane at z=1]
  \coordinate (c1) at (1.5,2.4);
  \coordinate (c2) at (4,3.5);
  \coordinate (c3) at ({4.9+2*sinh(\a)+cosh(\a)},{0.2+2*cosh(\a)+sinh(\a)});
\end{scope}

\path[facestyle] (b1) -- (b2) -- (b3) -- (b4) -- (b1);%hinten
\path[facestyle] (a1) -- (a2) -- (b2) -- (b1) -- (a1);%unten
\path[facestyle] (a1) -- (b1) -- (b4) -- (a4) -- (a1);%links
\path[facestyle] (a4) -- (a3) -- (b3) -- (b4) -- (a4) node[pos=0.4, left=0.2em] {1};%oben
\path[facestyle] (a2) -- (b2) -- (b3) -- (a3) -- (a2);%rechts
\path[facestyle] (a1) -- (a2)  node[pos=0.5, below] {1} -- (a3) -- (a4) -- (a1) node[pos=0.5, left] {1} ;%vorne

\path[facestyle] (b5) -- (b6) -- (b7) -- (b8) -- (b5);%hinten
\path[facestyle] (a5) -- (a6) -- (b6) -- (b5) -- (a5);%unten
\path[facestyle] (a5) -- (b5) -- (b8) -- (a8) -- (a5);%links
\path[facestyle] (a8) -- (a7) -- (b7) -- (b8) -- (a8) node[pos=0.4, left=0.2em] {1};%oben
\path[facestyle] (a6) -- (b6) -- (b7) -- (a7) -- (a6);%rechts
\path[facestyle] (a5) -- (a6)  -- (a7) -- (a8) -- (a5);%vorne

\path[draw=black] (a5) -- (ag) node[pos=0.49,below=-0.2em] {\tiny$\underbrace{\hspace{8.6em}}_{\cosh(\alpha)}$} -- (a6) node[pos=0.51, right=-0.3em] {\tiny$\left.\begin{matrix}\\[0.6em]\end{matrix}\right\}\sinh(\alpha)$};
%\path[draw=black] (a6) -- (a7) -- (ag) node[pos=0.51, right=-0.3em] {\tiny$\left.\begin{matrix}\\[8.7em]\end{matrix}\right\}\sqrt{\cosh(2\alpha)}$} -- (a6) node[pos=0.49,below=-0.2em] {\tiny$\underbrace{\hspace{5.5em}}_\frac{\sinh(2\alpha)}{\sqrt{\cosh(2\alpha)}}$};

\draw[->,thick] (c1) .. controls (c2) .. (c3) node[pos=0.5, above] {\large $V_\alpha$};

\path[draw=black] (a5) -- (ah);
%\draw (a5) pic{carc=-30:30:2em} ;
\draw (a5) ++(90:1.2) arc (90:{90-atan(tanh(\a))}:1.2) node[left=-0.1em]{$\overset{\phantom b}{\vartheta}$};
\end{scope}

%% file: tikz/tikz_DecompositionSimpleShear.tex
\input{tikz/deformationLibrary.tex}

\OmegaSetDefaults
\OmegaSetGridSize{.4}{.4}
\OmegaSetOutlineSampleCount{5}
\OmegaSetGridSampleCount{5}
\OmegaSetOutlineStyle{thick, color=black}
\OmegaSetShadingStyle{left color = lightgray, right color = black, opacity = .0}
\OmegaSetGridStyle{help lines}
\OmegaSetOutlineToRectangle{0}{0}{2}{2}
%\OmegaSetOutlineToUnitCircle

\OmegaSetDeformationToId
\OmegaSetSharpOutline
%\OmegaSetSmoothOutline
%\OmegaDraw

\definecolor{darkgreen}{rgb}{0.15,0.5,0.15}
\definecolor{darkred}{rgb}{1,0.15,0.15}

\def\a{0.4}
\def\g{0.6}
\begin{scope}[declare function={
        deformEpsx(\x,\y)=\x+\g/2*\y;
        deformEpsy(\x,\y)=\g/2*\x+\y;
        deformFx(\x,\y)=\x+\g*\y;
        deformFy(\x,\y)=\y;
 	}]

	\begin{scope}[xshift=0, yshift=0]

		\begin{scope}[xshift=0, yshift=0]
			\OmegaDraw
			\coordinate (a1) at (2.3,1.5);
			\coordinate (b1) at (2.3,0.5);
		\end{scope}

		\begin{scope}[xshift=140, yshift=0]
			\draw[lightgray,dashed] (0,0) -- (2,0) -- (2,2) -- (0,2) -- (0,0);
			\OmegaSetDeformation{deformEpsx}{deformEpsy}
			\OmegaDraw

			\coordinate (a2) at (-0.2,1.5);
  			\coordinate (a3) at (2.8,1.5);

		\end{scope}

		\begin{scope}[xshift=280, yshift=0]
			\draw[lightgray,dashed] (0,0) -- ({deformEpsx(2,0)},{deformEpsy(2,0)}) -- ({deformEpsx(2,2)},{deformEpsy(2,2)}) -- ({deformEpsx(0,2)},{deformEpsy(0,2)}) -- (0,0);

    			\OmegaSetDeformation{deformFx}{deformFy}
			\OmegaDraw
			\draw[darkred] (0,0) -- (0,2);
			%\draw[darkred] (0,1) arc  (90:90-atan(tanh(2*\a)):1) node[pos=0.35,below] {$\scriptsize\vartheta_l$};
			%\draw[darkred] (0,1) arc  (90:90-atan(\g):1) node[pos=0.35,below] {$\scriptsize\vartheta$};
			\draw[darkred] (0,2) -- ({2*\g},2) node[pos=0.5,above] {$\gamma$};

			\coordinate (a4) at (-0.2,1.5);
			\coordinate (b2) at (-0.2,0.5);

		\end{scope}

		\draw[->,thick] (a1) -- (a2) node[pos=0.5, above] {\large $\id+\eps_\gamma$};
		%\draw[->,thick] (a3) -- (a4) node[pos=0.5, above] {\large $\id+\omega+\cdots$};
		\draw[->,thick] (b1) to[out=-30,in=210] node[pos=1, below, yshift=-0.6cm,xshift=0.3cm] {\large $F_\gamma=\id+\eps_\gamma+\omega_\gamma$} (b2);

	\end{scope}

%\draw[->,thick] (-2.4,1.5) .. controls (-0.6,2.2) and (0.6,2.2) ..  (2.4,1.5) node[pos=0.5, above] {\large $F$};
\end{scope}

%% file: tikz/tikz_torsion.tex
    	\def\r{2.0}
    	\def\w{30}
    	\def\d{7}
		\pgfmathsetmacro\a{atan(1/\r)}
		\pgfmathsetmacro\hz{\r*sin(\a)}
		\pgfmathsetmacro\hy{\r*cos(\a)}
		\pgfmathsetmacro\xy{\r*sin(\w)}
		\pgfmathsetmacro\xz{-\r*cos(\w)}
		\pgfmathsetmacro\xyp{\r*sin(\w+\d)}
		\pgfmathsetmacro\xzp{-\r*cos(\w+\d)}
		\pgfmathsetmacro\xym{\r*sin(\w-\d)}
		\pgfmathsetmacro\xzm{-\r*cos(\w-\d)}
		
    	\begin{scope}[x  = {(-1cm,0cm)},
                    y  = {(0cm,1cm)},
                    z  = {(-0.5cm,0.5cm)},
                    scale = 1]
% style of faces
			\definecolor{lightgray2}{rgb}{0.95,0.95,0.95}
			\tikzset{facestyle/.style={fill=lightgray2,draw=black,thick,line join=round,opacity=0.65}}
			\tikzset{facestyle2/.style={fill=lightgray2,draw=white,thick,line join=round,opacity=0.65}}
	% face "back"
			\begin{scope}[canvas is yz plane at x=-1.5]
				\coordinate (a1) at (0,0);
				\draw[->,very thick] (\r/4,0) arc(0:-315:\r/4) node[above,xshift=0.3cm]{\begin{tabular}{c}Resisting \\ Torque \end{tabular}};
			\end{scope}		
			
			\begin{scope}[canvas is yz plane at x=0]
				%\coordinate (a1) at (0,0);
				\coordinate (a2) at (\hy,\hz);
				\coordinate (a3) at (-\hy,-\hz);
				\draw[facestyle] (-\r,-\r) grid (\r,\r);
				\draw[facestyle2] (0,0) circle (\r);
				\draw[dashed] (0,0) circle (\r);
				\draw (\hy,\hz) arc(\a:-180+\a:\r);
			\end{scope}
			
			\begin{scope}[canvas is yz plane at x=2.5]
				\draw[->,very thick] (\xy,\xz) arc(\w-90:-\w-90:\r);
			\end{scope}
			
			\begin{scope}[canvas is yz plane at x=3]
				\draw[dashed] (0,0) circle (\r);
				\draw (\hy,\hz) arc(\a:-180+\a:\r);
				\coordinate (x1) at (\xy,\xz);
				\coordinate (x2) at (-\xy,\xz);
				\coordinate (y1) at (\xym,\xzm);
				\coordinate (y2) at (-\xyp,\xzp);
				\path[facestyle] (x1) arc(\w-90:-\w-90:\r)--(-\xy-2,\xz+4) arc(-\w-90:\w-90:\r)--(x1);
			\end{scope}
			
			\begin{scope}[canvas is yz plane at x=5]
				\draw[dashed] (0,0) circle (\r);
				\draw (\hy,\hz) arc(\a:-180+\a:\r);
				\coordinate (x3) at (\xy,\xz);
				\coordinate (x4) at (-\xy,\xz);
				\coordinate (y3) at (\xyp,\xzp);
				\coordinate (y4) at (-\xym,\xzm);
			\end{scope}
			
			\begin{scope}[canvas is yz plane at x=5.5]
				\draw[->,very thick] (-\xy,\xz) arc(-\w-90:+\w-90:\r) node[below,yshift=-2.4cm,xshift=1.4cm] {Applied shear};
			\end{scope}
			
			\begin{scope}[canvas is yz plane at x=8]
				\coordinate (b1) at (0,0);
				\coordinate (b2) at (\hy,\hz);
				\coordinate (b3) at (-\hy,-\hz);
				\draw[facestyle] (b1) circle (\r);
				\draw[->,very thick] (\r*1.1,0) arc(0:90:\r*1.1);
				\draw[->,very thick] (-\r*1.1,0) arc(180:270:\r*1.1);
			\end{scope}
			
			\begin{scope}[canvas is yz plane at x=9.5]
				\coordinate (c1) at (0,0);
				\draw[->,very thick] (\r/4,0) arc(0:315:\r/4) node[above,yshift=0.3cm,xshift=-0.5cm]{\begin{tabular}{c}Applied \\ Torque \end{tabular}};
			\end{scope}
			
			\begin{scope}[canvas is yz plane at x=10]
				\coordinate (c1) at (0,0);
			\end{scope}
			
			\begin{scope}[canvas is yz plane at x=10.5]
				\coordinate (p1) at (0,0);
			\end{scope}
			
			\begin{scope}[canvas is yz plane at x=12]
				\coordinate (p2) at (0,0);
			\end{scope}
			
			\draw[->,dashdotted, thick] (a1)--(c1);
			\draw (a2)--(b2);
			\draw (a3)--(b3);
			\draw[->, very thick] (p1) -- (p2) node[xshift=0.8cm, below]{Poynting};
			\draw[very thick,dashed](y1)--(y3);
			\draw[very thick,dashed](y2)--(y4);

		\end{scope}

%% file: tikz/tikz_pureShearStretchLinearFinite.tex
\input{tikz/deformationLibrary.tex}

\OmegaSetDefaults
\OmegaSetGridSize{.4}{.4}
\OmegaSetOutlineSampleCount{5}
\OmegaSetGridSampleCount{5}
\OmegaSetOutlineStyle{thick, color=black}
\OmegaSetShadingStyle{left color = lightgray, right color = black, opacity = .0}
\OmegaSetGridStyle{help lines}
\OmegaSetOutlineToRectangle{0}{0}{2}{2}
%\OmegaSetOutlineToUnitCircle

\OmegaSetDeformationToId
\OmegaSetSharpOutline
%\OmegaSetSmoothOutline
%\OmegaDraw

\definecolor{darkred}{rgb}{1,0.15,0.15}
\definecolor{darkgreen}{rgb}{0.15,0.5,0.15}
\definecolor{darkblue}{rgb}{0.15,0.15,0.5}

\definecolor{gray1}{rgb}{0.2,0.2,0.2}
\definecolor{gray2}{rgb}{0.4,0.4,0.4}
\definecolor{gray3}{rgb}{0.6,0.6,0.6}
\definecolor{gray4}{rgb}{0.7,0.7,0.7}

\def\gammer{2*sin(atan(tanh(\a))}

\begin{scope}[declare function={
        deformPx(\x,\y)=cosh(\a)*\x+sinh(\a)*\y;
        deformPy(\x,\y)=sinh(\a)*\x+cosh(\a)*\y;
		deformEpsx(\x,\y)=\x+\gammer/2*\y;
		deformEpsy(\x,\y)=\y+\gammer/2*\x;
 	}]

\begin{scope}[xshift=0, yshift=0]

\begin{scope}[xshift=0, yshift=0]
	\draw[darkblue]  (0,3.7) node {infinitesimal pure shear stretch $\eps_\gamma$};
	\draw[lightgray,dashed] (0,0) -- (4,4);
	\draw[dashed] (0,0) -- (2,0) -- (2,2) -- (0,2) -- (0,0);
	\def\a{0.2}	
	\draw[gray1] (0,0) -- ({deformEpsx(2,0)},{deformEpsy(2,0)}) -- ({deformEpsx(2,2)},{deformEpsy(2,2)}) --({deformEpsx(0,2)},{deformEpsy(0,2)}) -- (0,0);
	\def\a{0.44}	
	\draw[gray2] (0,0) -- ({deformEpsx(2,0)},{deformEpsy(2,0)}) -- ({deformEpsx(2,2)},{deformEpsy(2,2)}) --({deformEpsx(0,2)},{deformEpsy(0,2)}) -- (0,0);
	\def\a{1.4}	
	\draw[gray3] (0,0) -- ({deformEpsx(2,0)},{deformEpsy(2,0)}) -- ({deformEpsx(2,2)},{deformEpsy(2,2)}) --({deformEpsx(0,2)},{deformEpsy(0,2)}) -- (0,0);

%    \OmegaSetDeformation{deformPx}{deformPy}
%	\OmegaDraw
\end{scope}

\begin{scope}[xshift=140, yshift=0]
	\draw[darkblue]  (0.9,3.7) node {finite pure shear stretch $V_\alpha$};
	\draw[lightgray,dashed] (0,0) -- (4.7,4.7);
	\draw[dashed] (0,0) -- (2,0) -- (2,2) -- (0,2) -- (0,0);
	\def\a{0.2}	
	\draw[gray1] (0,0) -- ({deformPx(2,0)},{deformPy(2,0)}) -- ({deformPx(2,2)},{deformPy(2,2)}) --({deformPx(0,2)},{deformPy(0,2)}) -- (0,0);
	\def\a{0.4}	
	\draw[gray2] (0,0) -- ({deformPx(2,0)},{deformPy(2,0)}) -- ({deformPx(2,2)},{deformPy(2,2)}) --({deformPx(0,2)},{deformPy(0,2)}) -- (0,0);
	\def\a{0.8}	
	\draw[gray3] (0,0) -- ({deformPx(2,0)},{deformPy(2,0)}) -- ({deformPx(2,2)},{deformPy(2,2)}) --({deformPx(0,2)},{deformPy(0,2)}) -- (0,0);

%    \OmegaSetDeformation{deformPx}{deformPy}
%	\OmegaDraw
\end{scope}

\end{scope}

\draw[->,thick] (2.3,1) to[out=30,in=150] node[pos=0.5, above] {\large $\exp$} (4.8,1);
\end{scope}

%% file: tikz/tikz_PolarDecompositionFiniteSimpleShear1.tex
\input{tikz/deformationLibrary.tex}

\OmegaSetDefaults
\OmegaSetGridSize{.4}{.4}
\OmegaSetOutlineSampleCount{5}
\OmegaSetGridSampleCount{5}
\OmegaSetOutlineStyle{thick, color=black}
\OmegaSetShadingStyle{left color = lightgray, right color = black, opacity = .0}
\OmegaSetGridStyle{help lines}
\OmegaSetOutlineToRectangle{0}{0}{2}{2}
%\OmegaSetOutlineToUnitCircle

\OmegaSetDeformationToId
\OmegaSetSharpOutline
%\OmegaSetSmoothOutline
%\OmegaDraw

\definecolor{darkgreen}{rgb}{0.15,0.5,0.15}
\definecolor{darkred}{rgb}{1,0.15,0.15}

\def\a{0.4}
\begin{scope}[declare function={
        deformPx(\x,\y)=cosh(\a)*\x+sinh(\a)*\y;
        deformPy(\x,\y)=sinh(\a)*\x+cosh(\a)*\y;
        deformRx(\x,\y)=(cosh(\a)*\x+sinh(\a)*\y)/sqrt(cosh(2*\a));
        deformRy(\x,\y)=(-sinh(\a)*\x+cosh(\a)*\y)/sqrt(cosh(2*\a));
		deformPRx(\x,\y)=deformPx(deformRx(\x,\y),deformRy(\x,\y));
		deformPRy(\x,\y)=deformPy(deformRx(\x,\y),deformRy(\x,\y));
		deformRPx(\x,\y)=deformRx(deformPx(\x,\y),deformPy(\x,\y));
		deformRPy(\x,\y)=deformRy(deformPx(\x,\y),deformPy(\x,\y));
		deformA(\x,\y)=1/sqrt(cosh(2*\a))*\x;
		deformB(\x,\y)=sqrt(cosh(2*\a))*\y;
 	}]

	\begin{scope}[xshift=0, yshift=0]

		\begin{scope}[xshift=0, yshift=0]
			\OmegaDraw
			\draw[blue,dashed] (1,0) arc (0:0-atan(tanh(\a)):1);
			\draw[blue,dashed] ({0.9*cos(atan(tanh(\a)))}, {-0.9*sin(atan(tanh(\a)))}) -- ({1.1*cos(atan(tanh(\a)))}, {-1.1*sin(atan(tanh(\a)))});
			\coordinate (b1) at (2.3,1.5);
			\coordinate (a1) at (2.3,0.5);
		\end{scope}

		\begin{scope}[xshift=140, yshift=50]
		\draw[lightgray,dashed] (0,0) -- (2,0) -- (2,2) -- (0,2) -- (0,0);
	    		\OmegaSetDeformation{deformA}{deformB}
			\OmegaDraw
			\coordinate (b2) at (-0.2,0.5);
			\coordinate (b3) at (2.3,0.5);

		\end{scope}
		
		\begin{scope}[xshift=140, yshift=-50]
			\draw[lightgray,dashed] (0,0) -- (2,0) -- (2,2) -- (0,2) -- (0,0);
	    		\OmegaSetDeformation{deformRx}{deformRy}
			\OmegaDraw
	  		\draw[blue,thick,->] (1,0) arc (0:0-atan(tanh(\a)):1);
	  		\draw[blue,thick] (0.9,0) -- (1.1,0);
			%  \draw[blue,dashed] (0,0) -- (0,2);
			%\draw[blue,dashed] (0,1) arc (90:90-atan(tanh(\a)):1);
			\draw[-|,dashed,darkgreen] (0,0) -- (2,2);
			\draw[|-|,dashed,darkgreen] (-1,1) -- (1,-1);  

			\coordinate (a2) at (-0.2,1.5);
			\coordinate (a3) at (2.8,1.5);

		\end{scope}

		\begin{scope}[xshift=280, yshift=0]
			\draw[lightgray,dashed] (0,0) -- ({deformRx(2,0)},{deformRy(2,0)}) -- ({deformRx(2,2)},{deformRy(2,2)}) -- ({deformRx(0,2)},{deformRy(0,2)}) -- (0,0);
	
	    		\OmegaSetDeformation{deformPRx}{deformPRy}
			\OmegaDraw
	  		\draw[darkred] (0,0) -- (0,2);
%			\draw[darkred] (0,1) arc  (90:90-atan(tanh(2*\a)):1) node[pos=0.35,below] {$\scriptsize\vartheta_l$};
			\draw[darkred] (0,2) -- ({2*tanh(2*\a)},2) node[pos=0.5,above] {$\gamma$};
			\draw[-|,dashed,darkgreen] (0,0) -- (2,2);
			\draw[|-|,dashed,darkgreen] (-1,1) -- (1,-1);  
			\draw[->,darkgreen] (0,0) -- ({deformPx(2,2)},{deformPy(2,2)});
			\draw[>-<,darkgreen]  ({deformPx(-1,1)},{deformPy(-1,1)}) -- ({deformPx(1,-1)},{deformPy(1,-1)});

			\coordinate (b4) at (-0.2,1.5);
			\coordinate (a4) at (-0.2,0.5);

		\end{scope}
		\draw[->,thick] (a1) to[out=-30,in=180] node[pos=0.5, above] {\large $R$} (a2) ;
		\draw[->,thick] (a3) to[out=0,in=210] node[pos=0.5, above] {\large $V_\alpha$} (a4);
		\draw[->,thick] (b1) to[out=30,in=180] node[pos=0.6, below,yshift=-0.3cm] {\large $\diag(a,b,c)$} (b2);
		\draw[->,thick] (b3) to[out=0,in=150] node[pos=0.5, below] {\large $F_\gamma$} (b4);
	
	\end{scope}

%\draw[->,thick] (-2.4,1.5) .. controls (-0.6,2.2) and (0.6,2.2) ..  (2.4,1.5) node[pos=0.5, above] {\large $F$};
\end{scope}

%% file: tikz/tikz_shearTabular.tex
%\begin{scope}[scale=2.2]
\begin{scope}

	\definecolor{darkred}{rgb}{1,0.15,0.15}
	\definecolor{darkgreen}{rgb}{0.15,0.5,0.15}
	\definecolor{darkblue}{rgb}{0.15,0.15,0.5}

	%\tikzstyle arrowstyle=[scale=27]
	%%  draw axes
	%\draw[color=black, very thin, -latex] (-1.1,0) -- (1.2,0) node[below] {$\textrm{Re}$};
	%\draw[color=black, very thin, -latex] (0,-1.1) -- (0,1.25) node[below right] {$\textrm{Im}$};
	\tikzset{facestyle/.style={fill=lightgray,draw=black,thick,line join=round}}

	\begin{scope}[xshift=0, yshift=0]
  
		\draw[darkblue]  (0,1.5) node {\begin{tabular}{c}infinitesimal \\ pure shear stress\end{tabular}};
		\draw  (0,0) node {$\matr{0 & s & 0 \\[.5em] s & 0 & 0 \\[.5em] 0 & 0 & 0}$};
		\draw  (0,-1.2) node {$\sigma$};
		\draw[darkblue]  (9,1.5) node {infinitesimal pure shear strain};
		\draw  (9,0) node {$\matr{ 0 & \frac \gamma 2 & 0 \\[.5em] \frac \gamma 2 & 0 & 0 \\[.5em] 0 & 0 & 0}$};
		\draw  (9,-1.2) node {$\eps_\gamma$};		
		\draw[darkblue]  (15,1.5) node {linear simple shear deformation};
		\draw  (15,0) node {$\matr{1 & \gamma & 0 \\[.5em] 0 & 1 & 0 \\[.5em] 0 & 0 & 1} $};
		\draw  (15,-1.2) node {$F_\gamma$};

		\draw[thick,->] (1.2,0) -- (7.7,0);
		\draw[thick,->] (10.3,0) -- (13.7,0);
	\end{scope}

	\begin{scope}[xshift=0, yshift=-105]
  
		\draw[darkblue]  (0,1.5) node {pure shear stress};
		\draw  (0,0) node {$\matr{0 & s & 0 \\[.5em] s & 0 & 0 \\[.5em] 0 & 0 & 0}$};
		\draw  (0,-1.2) node {$\sigma$};
		\draw  (4,0) node {$\matr{ \frac 1{\lambda_1 + \lambda_2} & \frac 1{\lambda_1 - \lambda_2} & 0 \\[.5em] \frac 1{\lambda_1 - \lambda_2} & \frac 1{\lambda_1 + \lambda_2} & 0 \\[.5em] 0 & 0 & \lambda_3}$};
		\draw  (4,-1.2) node {$V$};
		\draw[darkblue]  (9,1.5) node {finite pure shear stretch};
		\draw  (9,0) node {$\matr{\cosh(\alpha) & \sinh(\alpha) & 0 \\[.5em] \sinh(\alpha) & \cosh(\alpha) & 0 \\[.5em] 0 & 0 & 1}$};
		\draw  (9,-1.2) node {$V_\alpha$};
		\draw[darkblue]  (15,1.5) node {left finite simple shear deformation};
		\draw  (15,0) node {$\begin{pmatrix} \frac 1{\sqrt{\cosh(2\alpha)}} & \frac{\sinh(2\alpha)}{\sqrt{\cosh(2\alpha)}} & 0 \\[.5em] 0 & \sqrt{\cosh(2\alpha)} & 0 \\[.5em] 0 & 0 & 1 \end{pmatrix}$};
		\draw  (15,-1.2) node {$F_\alpha$};
	
		\draw[thick,->] (1.2,0) -- (2.0,0);
		\draw[darkgreen,thick,->] (5.9,0) -- (6.9,0);
		\draw[darkgreen,thick,->] (11.1,0) -- (12.4,0);

	\end{scope}

		\fill [ fill = darkgreen ,  opacity = .07] (5.7 ,-3.4) rectangle (7.1 ,-7.65) ;
%		\draw[darkgreen]  (6.4,-5.2) node {$\begin{matrix}\det F = 1\\F.e_3=e_3\end{matrix}$};
		\draw[darkgreen]  (6.4,-5.2) node {$\begin{matrix}\det V = 1\,,\\\text{planar}\end{matrix}$};

	\begin{scope}[xshift=0, yshift=-210]
  
		\draw[darkblue]  (0,1.5) node {pure shear stress};
		\draw  (0,0) node {$\matr{0 & s & 0 \\[.5em] s & 0 & 0 \\[.5em] 0 & 0 & 0}$};
		\draw  (0,-1.2) node {$\TBiot, S_2$};
		\draw  (4,0) node {$\matr{ \frac 1{\lambda_1 + \lambda_2} & \frac 1{\lambda_1 - \lambda_2} & 0 \\[.5em] \frac 1{\lambda_1 - \lambda_2} & \frac 1{\lambda_1 + \lambda_2} & 0 \\[.5em] 0 & 0 & \lambda_3}$};
		\draw  (4,-1.2) node {$U$};
		\draw[darkblue]  (9,1.5) node {finite pure shear stretch};
		\draw  (9,0) node {$\matr{\cosh(\alpha) & \sinh(\alpha) & 0 \\[.5em] \sinh(\alpha) & \cosh(\alpha) & 0 \\[.5em] 0 & 0 & 1}$};
		\draw  (9,-1.2) node {$U_\alpha$};
		\draw[darkblue]  (15,1.5) node {right finite simple shear deformation};
		\draw  (15,0) node {$\begin{pmatrix} \sqrt{\cosh(2\alpha)} & \frac{\sinh(2\alpha)}{\sqrt{\cosh(2\alpha)}} & 0 \\[.5em] 0 & \frac 1{\sqrt{\cosh(2\alpha)}} & 0 \\[.5em] 0 & 0 & 1 \end{pmatrix}$};
		\draw  (15,-1.2) node {$F_\alpha$};
		
		\draw[thick,->] (1.2,0) -- (2.0,0);
		\draw[darkgreen,thick,->] (5.9,0) -- (6.9,0);
		\draw[darkgreen,thick,->] (11.1,0) -- (12.4,0);

	\end{scope}
	
		\fill [ fill = darkgreen ,  opacity = .07] (10.9 ,-3.4) rectangle (12.6 ,-7.65) ;
%		\draw[darkgreen]  (6.4,-5.2) node {$\begin{matrix}\det F = 1\\F.e_3=e_3\end{matrix}$};
		\draw[darkgreen]  (11.7,-5.2) node {$\begin{matrix}\text{ground-}\\\text{parallel}\end{matrix}$};

\end{scope}

%% file: tikz/tikz_PolarDecompositionFiniteSimpleShear2.tex
\input{tikz/deformationLibrary.tex}

\OmegaSetDefaults
\OmegaSetGridSize{.4}{.4}
\OmegaSetOutlineSampleCount{5}
\OmegaSetGridSampleCount{5}
\OmegaSetOutlineStyle{thick, color=black}
\OmegaSetShadingStyle{left color = lightgray, right color = black, opacity = .0}
\OmegaSetGridStyle{help lines}
\OmegaSetOutlineToRectangle{0}{0}{2}{2}
%\OmegaSetOutlineToUnitCircle

\OmegaSetDeformationToId
\OmegaSetSharpOutline
%\OmegaSetSmoothOutline
%\OmegaDraw

\definecolor{darkgreen}{rgb}{0.15,0.5,0.15}
\definecolor{darkred}{rgb}{1,0.15,0.15}

\def\a{0.4}
\begin{scope}[declare function={
        deformPx(\x,\y)=cosh(\a)*\x+sinh(\a)*\y;
        deformPy(\x,\y)=sinh(\a)*\x+cosh(\a)*\y;
        deformRx(\x,\y)=(cosh(\a)*\x+sinh(\a)*\y)/sqrt(cosh(2*\a));
        deformRy(\x,\y)=(-sinh(\a)*\x+cosh(\a)*\y)/sqrt(cosh(2*\a));
		deformPRx(\x,\y)=deformPx(deformRx(\x,\y),deformRy(\x,\y));
		deformPRy(\x,\y)=deformPy(deformRx(\x,\y),deformRy(\x,\y));
		deformRPx(\x,\y)=deformRx(deformPx(\x,\y),deformPy(\x,\y));
		deformRPy(\x,\y)=deformRy(deformPx(\x,\y),deformPy(\x,\y));
		deformA(\x,\y)=\x+tanh(2*\a)*\y;
		deformB(\x,\y)=\y;
 	}]

	\begin{scope}[xshift=0, yshift=0]

		\begin{scope}[xshift=0, yshift=0]
			\OmegaDraw
			\draw[-|,dashed,darkgreen] (0,0) -- (2,2);
			\draw[|-|,dashed,darkgreen] (-1,1) -- (1,-1);  
			\coordinate (b1) at (2.3,1.5);
			\coordinate (a1) at (2.3,0.5);
		\end{scope}
		
		\begin{scope}[xshift=140, yshift=50]
		\draw[lightgray,dashed] (0,0) -- (2,0) -- (2,2) -- (0,2) -- (0,0);
	    		\OmegaSetDeformation{deformA}{deformB}
			\OmegaDraw
			\coordinate (b2) at (-0.2,0.5);
			\coordinate (b3) at (2.3,0.5);

		\end{scope}

		\begin{scope}[xshift=140, yshift=-50]
			\draw[lightgray,dashed] (0,0) -- (2,0) -- (2,2) -- (0,2) -- (0,0);

			\OmegaSetDeformation{deformPx}{deformPy}
			\OmegaDraw
			%\draw[<->,thick] (0,0) -- ({deformPx(2,2)},{deformPy(2,2)});
			%\draw[>-<,thick]  ({deformPx(0,2)},{deformPy(0,2)}) -- ({deformPx(2,0)},{deformPy(2,0)});
			\draw[-|,dashed,darkgreen] (0,0) -- (2,2);
			\draw[|-|,dashed,darkgreen] (-1,1) -- (1,-1);  
			\draw[->,darkgreen] (0,0) -- ({deformPx(2,2)},{deformPy(2,2)});
			\draw[>-<,darkgreen]  ({deformPx(-1,1)},{deformPy(-1,1)}) -- ({deformPx(1,-1)},{deformPy(1,-1)});

			%  \draw[blue,dashed] (0,0) -- (0,2);
			%\draw[blue,dashed] (0,1) arc (90:90-atan(tanh(\a)):1);
			\draw[blue,dashed] ({cos(atan(tanh(\a)))}, {sin(atan(tanh(\a)))}) arc (atan(tanh(\a)):0:1);
			\draw[blue,dashed] (0.9,0) -- (1.1,0);

			\coordinate (a2) at (-0.2,1.5);
			\coordinate (a3) at (2.8,1.5);
		\end{scope}

		\begin{scope}[xshift=280, yshift=0]
			\draw[lightgray,dashed] (0,0) -- ({deformPx(2,0)},{deformPy(2,0)}) -- ({deformPx(2,2)},{deformPy(2,2)}) -- ({deformPx(0,2)},{deformPy(0,2)}) -- (0,0);

			\OmegaSetDeformation{deformRPx}{deformRPy}
			\OmegaDraw
			\draw[darkred] (0,0) -- (0,2);
			\draw[darkred] (0,2) -- ({2*sinh(2*\a)},2) node[pos=0.4,above] {$\gamma$} -- ({deformRPx(0,2)},{deformRPy(0,2)});

%			\draw[darkred] (0,1) arc (90:90-atan(sinh(2*\a)):1)  node[pos=0.4,below] {$\scriptsize\vartheta_r$};
			\draw[blue,thick,->] ({cos(atan(tanh(\a)))}, {sin(atan(tanh(\a)))}) arc (atan(tanh(\a)):0:1);
			\draw[blue,thick] ({0.9*cos(atan(tanh(\a)))}, {0.9*sin(atan(tanh(\a)))}) -- ({1.1*cos(atan(tanh(\a)))}, {1.1*sin(atan(tanh(\a)))});

			\coordinate (b4) at (-0.2,1.5);
			\coordinate (a4) at (-0.2,0.5);

		\end{scope}
		\draw[->,thick] (a1) to[out=-30,in=180] node[pos=0.5, above] {\large $U_\alpha$} (a2) ;
		\draw[->,thick] (a3) to[out=0,in=210] node[pos=0.5, above] {\large $R$} (a4);
		\draw[->,thick] (b1) to[out=30,in=180] node[pos=0.5, below] {\large $F_\gamma$} (b2);
		\draw[->,thick] (b3) to[out=0,in=150] node[pos=0.4, below,yshift=-0.3cm] {\large $\diag(a,b,c)$} (b4);

	\end{scope}

%\draw[->,thick] (-2.4,1.5) .. controls (-0.6,2.2) and (0.6,2.2) ..  (2.4,1.5) node[pos=0.5, above] {\large $F$};
\end{scope}